\documentclass[11pt]{article}
\usepackage{amsmath,amsthm,amsfonts,amssymb,bm,mathrsfs,wasysym}
\usepackage{epsfig}
\usepackage[usenames]{color}
\usepackage{verbatim}
\usepackage{hyperref}
\usepackage{multicol}
\usepackage{comment}
\usepackage{float}
\usepackage{graphicx}
\usepackage[utf8]{inputenc}

\parskip \medskipamount
\newtheorem{theorem}{Theorem}[section]

\numberwithin{equation}{section}
\newtheorem{lemma}[theorem]{Lemma}
\newtheorem{proposition}[theorem]{Proposition}
\newtheorem{corollary}[theorem]{Corollary}

\newtheorem{remark}[theorem]{Remark}

\numberwithin{equation}{section}

\def\Z{\mathbb{Z}}

\def\T{\mathcal{T}}

\def\S{\mathcal{S}}

\def\bE{\mathbb{E}}

\allowdisplaybreaks

\newcommand{\1}{{\text{\Large $\mathfrak 1$}}}

\def\bs{\backslash}

\renewcommand{\emptyset}{\varnothing}
\renewcommand{\phi}{\varphi}
\renewcommand{\epsilon}{\varepsilon}
\def\tilde{\widetilde}

\def\reff#1{(\ref{#1})}

\begin{document}
\title{\bf Time spent in a ball by a critical branching random walk}

\author{Amine Asselah \thanks{
LAMA, Univ Paris Est Creteil, Univ Gustave Eiffel, UPEM, CNRS, F-94010, Cr\'eteil, France \& New York University at Shanghai; amine.asselah@u-pec.fr} \and
Bruno Schapira\thanks{Aix-Marseille Universit\'e, I2M, CNRS UMR 7373, 13453 Marseille, France;  \&  Universit\'e Lyon 1, Institut Camille Jordan, CNRS UMR 5208, ; bruno.schapira@univ-amu.fr} 
}
\date{}
\maketitle
\begin{abstract}
We study a critical branching random walk on $\Z^d$.
We focus on the tail of the time spent in a ball, 
and our study, in dimension four and higher, sheds new light on the
recent result of Angel, Hutchcroft and Jarai \cite{AHJ}, 
in particular on the special features of the critical dimension four.
Finally, we analyse the number of walks transported 
by the branching random walk on the boundary of a distant ball. \\

MSC 2010 \emph{subject classifications.} Primary 60G50; 60J80.
\vspace{0.2cm}
\end{abstract}

\section{Introduction}\label{sec-intro}
In this paper we study a critical branching random walk (BRW) on $\Z^d$.
Whereas the study of the volume of the range of random walks is a central 
object of probability theory, the range of branching random walks,
 in dimension larger than one, stayed in the shadows. 
Quite recently, Le Gall and Lin \cite{LGL1,LGL2} proved limit theorems
for the volume of the range, say $\mathcal R_n$, of a random walk indexed by a Galton-Watson tree 
conditioned on having $n$ vertices, as $n$ goes
to infinity. 
In particular they discovered that in dimension five and larger, 
$\mathcal R_n$
scales linearly, whereas in dimension four it scales like
$n/\log(n)$, and in dimension three and lower it scales 
like $n^{d/4}$. Thus with BRW one recovers the well-known trichotomy for the asymptotic behavior of the range 
of a simple random walk, going back to Dvoretzky and Erd\"os \cite{DE}, except that 
the critical dimension is now equal to four instead of two. 
Later, Zhu in a series of works~\cite{Zhu1,Zhu2,Zhu4,Zhu3} extended part of Le Gall and Lin's analysis to general offspring and jump distributions, 
and most notably brought into light   
the notion of {\it branching capacity}, which is associated with 
BRW just as the {\it electrostatic capacity}
is associated with random walk.  
Lalley and Zheng \cite{LZ11}, 
analyzed the occupation statistics at a fixed generation of the tree, 
and observed similar behaviors as for a simple random walk. 
Angel, Hutchcroft and Jarai in \cite{AHJ} 
studied the tail of the local times for the full tree, and 
discovered that the tail speed is 
exponential above the critical dimension (dimensions five and higher), but
stretched exponential in the critical dimension four, a fact which does not have natural counterpart in the random walk setting. 
One of our motivation for the present paper 
is to bring some light on this remarkable observation, and in particular 
explain the tail behaviour in dimensions four and higher.
When \cite{AHJ} follows a moment method, rooted in statistical mechanics,
our approach is probabilistic, and aims at developing an analogue of
excursion theory so useful to analyse random walks. Finally, to
emphasize the recent vigor of BRW studies in high dimensions,
let us mention \cite{BC12,LSS24} dealing with recurrence and transience of a discrete snake, some recent results 
on the electrostatic capacity of the range of a BRW \cite{BW20,BH21,BH22}, 
and others on Branching interlacements \cite{PZ19,Zhu18}, or on the range of tree-valued BRWs \cite{DKLT}.

Here, we consider one Euclidean ball centered at the origin, and study 
two objects:
(i) the tail of the time spent in this ball when the 
BRW starts at the origin;
(ii) the tail of the (rescaled) number of walks hitting the ball, 
when the BRW starts from far away.

To state our results, let us introduce the needed notation. We let $\mathbf {\mathcal T}$ be a critical 
Bienaym\'e-Galton-Watson tree (BGW tree for short), 
whose offspring distribution has a finite exponential moment. 
Consider $\{S_u, u\in \mathbf {\mathcal T}\}$ 
an associated tree-indexed random walk, 
which we view alternatively as a branching random walk, where
time is encoded by the tree $\mathcal T$ and
whose jump distribution is the uniform measure $\theta$ on the neighbors of the origin.
In other words, independent
increments $\{X(e)\}$ are associated to the edges of the tree, and if
$[\emptyset,u]$ is the sequence of edges between the root $\emptyset$ and
vertex $u$, then
\[
S_u=S_\emptyset+\sum_{e\in [\emptyset,u]}X(e).
\]
When $z\in \Z^d$, we let $\mathbb P_z$ be the law of the BRW
starting from $z$, i.e. conditioned on $\{S_\emptyset=z\}$, and simply write $\mathbb P$ when it starts from the origin. 
Given $\Lambda\subset \Z^d$, we define the time spent in $\Lambda$ by the BRW as 
\begin{equation}\label{def-times}
\ell_{\T}(\Lambda) := \sum_{u\in \mathcal T}
\1\{S_u \in \Lambda\}.  
\end{equation}
Let $B_r:=\{z\in \Z^d : \|z\|< r\}$, where $\|\cdot \|$ 
denotes the Euclidean distance, and
write $\Theta(f(t,r))$ for a function which is uniformly 
bounded from above and below by $f$, up to 
multiplicative positive constants (that may depend on the dimension).  Our main result reads as follows.

\begin{theorem} \label{theo.main}
One has uniformly in $r\ge 1$, and $t\ge 1$, 
\begin{equation}\label{cost.maintheo}
\mathbb P(\ell_{\T}(B_r) > t) = 
\Theta\Big(\frac{1}{\sqrt{\min(t,r^4)}}\Big) \times \left\{
\begin{array}{ll}
\exp \Big(-\Theta(t/r^4)\Big) & \text{if }d\ge 5 \\
\exp\Big(- \Theta( \sqrt{ t/r^4})\Big) & \text{if }d=4\\
(1+\frac t{r^4})^{-\frac{2}{4-d}} & \text{if }d=1,2,3.\\ 
\end{array}
\right.
\end{equation}
\end{theorem}
Note that when $t$ is of order $r^d$, then in the exponential we obtain a factor which is of order $r^{d-4}$, in dimension $5$ and higher, that is of the same order as the branching capacity of the ball $B_r$, as shown in~\cite{Zhu1}. 
This is not merely a coincidence, and a more general result in this direction has been shown in the recent paper~\cite{ASS23}. 
On the other hand, it remains an interesting open problem to prove the existence of a limiting constant, in front of the $t/r^4$ in the exponential. 
\begin{remark}
\emph{In fact the result in dimension $1$ and $2$ holds under the weaker assumption that the offspring distribution of the BGW tree has only a finite second moment, and a finite third moment in dimension three, instead of a finite exponential moment. Moreover, the proof in dimension $1,2,3$ can be done the same way as in \cite{AHJ}, using a simple moment method. 
}
\end{remark}

\paragraph{Heuristics.}
Let us explain at a heuristic level the difference between
dimension four on one hand and five and higher on the other hand.
In the latter case, to occupy $B_r$, a good strategy for the BRW is to produce 
{\bf waves}, which go from the boundary $\partial B_r$, 
up to the boundary of a larger concentric ball, 
say $\partial B_{2r}$, and back to $B_r$. 
Furthermore, (i) each wave starting with order $r^2$ BRWs 
hits the other boundary with order $r^2$ particles, at a constant cost; and 
(ii) when a single BRW starts on $\partial B_r$, 
it spends typically a time $r^2$ in $B_r$.
Thus $t/r^4$ waves typically produce a local time $t$ in $B_r$, and
the cost is that of creating these $t/r^4$ waves.

In dimension four, the situation is drastically different: conditioned
on coming from a distance $R$, a BRW typically brings $r^2\cdot \log(R/r)$ particles on
the boundary of $B_r$, so there is some advantage to coming from
far away (the $\log(R/r)$ factor is absent in $d\ge 5$). This can be seen
by saying that the process conditioned on hitting $B_r$ has a single
spine from which critical BRWs grow. We show that a {\it correct} scenario
(see Figure ~\ref{fig:dessin-d4} in Section~\ref{sec-LB4}) 
consists in producing $\log(R/r)$ spines, bringing a total number
of particles of order $r^2\cdot \log^2(R/r)$ and we equate this order
with $t/r^2$ to get the desired stretched exponential cost.

We note that the idea of using waves to study the tail of the local time distribution was also found useful in the recent paper~\cite{BHJ23}, which studies the thick points of branching Brownian motion and branching random walks.

Theorem~\ref{theo.main} generalizes the analysis of \cite{AHJ},
and its proof follows a probabilistic method. In this proof, we
encounter many interesting objects whose exponential moments
are studied, and they do present interest on their own.
In order to present these additional results, we now define more objects.
For a subset $\Lambda\subset \Z^d$, let $\mathcal T(\Lambda)$
be the set of vertices of $\mathcal T$ at which the BRW is in $\Lambda$, as well as at all its ancestral positions. In other words, 
\begin{equation}\label{def.Tr}
\mathbf {\mathcal T}(\Lambda) := 
\big\{u\in \mathbf{\mathcal T}: S_v\in \Lambda, \text{ for all }v\le u\},
\end{equation}
where we write $v\le u$ if $v$ is an ancestor of $u$ (by which we include the case $v=u$). 
Also, the set of vertices corresponding to hitting times of $\Lambda$, called here frozen particles, is denoted as 
\begin{equation}\label{etar}
\eta(\Lambda) : = \big\{u\in \mathbf {\mathcal T}: S_u \in \Lambda 
\text{ and }S_v \notin \Lambda \text{ for all }v< u\}, 
\end{equation}
where by $v<u$ we mean that $v$ is an ancestor of $u$, which is different from $u$. 
Our main new estimates concern the exponential moments of the 
number of frozen particles during each wave. There are two distinct
problems as whether we deal with starting points inside the ball, or outside it;  
the latter problem being the technical core of this paper.

Our first result concerns the case of a starting point 
lying inside the ball $B_r$, and 
we freeze walks as they reach its outer boundary $\partial B_r$. The result holds true irrespective of the dimension. 
For simplicity we write $\eta_r:=\eta(\partial B_r)$. 
 
\begin{theorem}\label{theo.inside}
Assume $d\ge 1$. There exist positive constants $c$ and $\lambda_0$ (only depending on the dimension), such that for any $0\le \lambda\le \lambda_0$, and any $r\ge 1$, 
\begin{equation}\label{inside2}
\sup_{x\in B_r}\, \mathbb E_x\Big[\exp(\lambda \frac{|\eta_r|}{r^2})\Big]
 \le \exp(\frac{c\lambda}{r^2}). 
\end{equation}
Moreover, 
\begin{equation}\label{inside1}
\sup_{x\in B_r} \, \mathbb E_x\Big[\exp(\lambda \frac{|\eta_r|}{r^2})\ \big| \ \eta_r\neq \emptyset\Big] \le \exp(c\lambda), 
\end{equation}
and 
\begin{equation}\label{inside3}
\sup_{x\in B_{r/2}} \, \mathbb E_x\Big[\exp(\lambda \frac{|\eta_r|}{r^2})
\Big] \le \exp\big( \frac{\lambda + c\lambda^2}{r^2}\big). 
\end{equation}
\end{theorem}
Actually~\eqref{inside1} easily follows from~\eqref{inside2}, once we know that the probability for the BRW to exit $B_r$ starting from a point
inside $B_r$, is {\it at least} of order $1/r^2$, see Lemma~\ref{lem.hit.inside}. On the other hand, \eqref{inside3} follows from \eqref{inside1}, once we 
know that the former probability, starting from a point in $B_{r/2}$ is 
{\it at most} of order $1/r^2$, see Proposition~\ref{prop.hit.inside}.

Thus the heart of the matter is to prove \eqref{inside2}. 
For this, we first show an analogous estimate for 
the total time spent inside the ball $B_r$ by a BRW killed on its boundary, 
whose set of particles is given by $\mathcal T(B_r)$, 
see \eqref{def.Tr} and Proposition~\ref{prop.Tr} below. 
Then we observe that $|\eta_r|$ and $|\mathcal T(B_r)|$ 
are linked via some natural martingales whose exponential moments are 
controlled by those of $|\mathcal T(B_r)|$. 
\begin{proposition}\label{prop.Tr}
Assume $d\ge 1$. There exist positive constants $c$ and $\lambda_0$, 
such that for any $0\le \lambda\le \lambda_0$, and any $r\ge 1$, 
$$
\sup_{x\in B_r}\ \mathbb E_x\Big[\exp\big(\lambda\frac{ 
|\mathbf{\mathcal T}(B_r)|}{r^4}\big)\Big]
\le\exp(c\frac{\lambda}{r^2}). 
$$ 
\end{proposition}
In fact we prove a slightly stronger result 
in Proposition~\ref{prop.Tr+}, which provides an important 
additional factor $1/r^2$ in the tail distribution, 
needed in the proof of Theorem~\ref{theo.main}. 

To conclude the proof of Theorem~\ref{theo.main}, 
we need to control the exponential moments of $|\eta_r|$, 
when starting from a point outside $B_r$. 
This part is delicate, and this is where the role of the dimension 
comes into play. Roughly, the reason for this is that 
it could happen that many particles would freeze 
on $\partial B_r$, only after doing very large excursions away from it. 
In dimension five and higher the price to pay for 
these large excursions is too expensive for playing a significant role. 
As a consequence one can deduce 
a result which is similar to Theorem~\ref{theo.inside}.
Let us however emphasize the factor $(1-\varepsilon)$ 
appearing in~\eqref{outside3}, which 
comes from the transience of the random walk in 
dimension three and higher, 
and which guarantees that only a finite number of waves matter.

\begin{theorem}\label{theo.outside5}
Assume $d\ge 5$. 
There exist positive constants $c$, $r_0$ and $\lambda_0$ (only depending on the dimension), such that for any $0\le \lambda\le \lambda_0$, any $r\ge r_0$, 
and $x\in B_{2r}^c$,
\begin{equation}\label{outside1}
\mathbb E_x\Big[\exp(\lambda \frac{|\eta_r|}{r^2})\Big] \le \exp(\frac{c\lambda}{\|x\|^2}). 
\end{equation}
As a consequence,
\begin{equation}\label{outside2}
\sup_{x\in \partial B_{2r}} \mathbb E_x\Big[\exp(\lambda \frac{|\eta_r|}{r^2})
\ \Big|\ \eta_r\neq \emptyset \Big] \le \exp(c\lambda), 
\end{equation}
and there exists $\varepsilon\in (0,1)$, such that for any $r\ge r_0$, and any $0\le \lambda\le \lambda_0$, 
\begin{equation}\label{outside3}
\sup_{x\in \partial B_{2r}} \mathbb E_x\left[\exp\big(\lambda \frac{ 
|\eta_r|}{r^2}\big)\right] \le \exp\big(\frac{\lambda (1-\varepsilon)}{r^2}\big).
\end{equation}
\end{theorem}
We note that the restriction $r\ge r_0$ in the above theorem could be dropped and replaced by $r\ge 1$, at the cost of some mild additional work, but since we shall not need it, we refrain from giving more details.

In dimension four, the situation is more subtle, and large excursions start to play a decisive role.  
In particular, one can show that 
all exponential moments of $|\eta_r|/r^2$ are infinite, when starting for instance from $\partial B_{2r}$. 
Thus one needs to consider instead a truncated version of $\eta_r$ and renormalize it conveniently. 
We do this here, by killing the BRW once it reaches some 
large distance. To formulate our result, we define the deposition on $B_r$ of
trajectories which remain in $B_R$ for $R>2r$:
\begin{equation}\label{etarR}
\eta_{r,R} = \eta_r\cap \T(B_R).
\end{equation}
In other words $\eta_{r,R}$ is the set of vertices 
of $\mathcal T(B_R)$ corresponding to hitting times of $\partial B_r$. Then, we obtain the following. 
\begin{theorem}\label{theo.outside4}
Assume $d = 4$. There exist positive constants $c$, $r_0$ and $\lambda_0$, such that for any $0\le \lambda\le \lambda_0$, any $r\ge r_0$, $R\ge 2r$,  
and all $x\in B_R\setminus B_{2r}$,
\begin{equation}\label{outside4.1}
\mathbb E_x\Big[\exp \big( \frac{\lambda|\eta_{r,R}|}{r^2\log (R/r)} \big)\Big] \le \exp\big(\frac{c\lambda}{\|x\|^2\log (R/r)}\big). 
\end{equation}
Furthermore, if $R\ge 4r$,
\begin{equation}\label{outside4.2}
\sup_{x\in \partial B_{2r}} \mathbb E_x\Big[
\exp \big( \frac{\lambda|\eta_{r,R}|}{r^2\log (R/r)} \big)\ \Big|\ 
\eta_{r,R}\neq \emptyset \Big] \le \exp(c\lambda),  
\end{equation}
and there exists $\varepsilon \in (0,1)$, such that 
\begin{equation}\label{outside4.3}
\sup_{x\in \partial B_{2r}} \mathbb E_x\left[\exp\big( 
\frac{\lambda |\eta_{r,R}|}{r^2\log(R/r)}\big)\right] \le \exp\big(\frac{\lambda(1-\varepsilon)}{r^2\log(R/r)}\big).
\end{equation}
\end{theorem}
Here as well, we note that the restriction $r\ge r_0$ could be dropped and replaced by $r\ge 1$.

These estimates allow to consider starting points which 
are not contained in the ball (equivalently balls not centered at the origin). 
For instance when $d\ge 5$, then uniformly in $r\ge 1$, $t\ge 1$, and $\|x\|\ge 2r$, 
\begin{equation*}
\mathbb P_x(\ell_{\T}(B_r) > t) = 
\Theta\big(\frac{r^{d-4}}{\|x\|^{d-2}}\big) \times\exp \big(-\Theta(t/r^4)\big). 
\end{equation*}
Similar estimates could be proved when $\|x\|\le 2r$, 
depending on the value of $t$, and the same could also be done in lower dimension. 

The rest of the paper is organized as follows.
Section~\ref{sec-notation} sets the notation and
recall some basic results.
Section~\ref{sec-prelim} deals with some moment bounds for 
a Bienaym\'e-Galton-Watson process. We also 
recall there the spine decomposition of the BRW obtained by Zhu, 
and give bounds for the small moments of 
$|\eta_r|$ both when the starting point lies inside and outside the ball. 
Section~\ref{sec-LD} deals with Theorem~\ref{theo.main} in low dimensions,
and Section~\ref{sec-Tr} deals with the 
exponential moments of the size of the localized BRW, as 
presented in Proposition~\ref{prop.Tr}. Theorem~\ref{theo.inside} is proved
in Section~\ref{sec-inside}. The technical heart of the paper
spreads over three sections: Theorem~\ref{theo.outside5} dealing with $d\ge 5$ is proved in
Section~\ref{sec-outside5}, Theorem~\ref{theo.outside4} dealing with $d=4$ is proved in Section~\ref{sec-outside4}, and the conclusion of the proof of the upper bounds in Theorem~\ref{theo.main} for $d\ge 4$ is explained in Section~\ref{sec-UB4}. Finally, the lower
bounds in high dimensions are given in Section~\ref{sec-LB4}.

\section{Notation and basic tools}\label{sec-notation}

We let $\mathbf{\mathcal T}$ be a Bienaym\'e-Galton-Watson tree (BGW for short), 
with offspring distribution some measure $\mu$ on the set of integers. Throughout the paper we assume that $\mu$ is critical, in the sense that its mean is equal to one, and that it has a finite variance, which we denote by $\sigma^2$. 
When dimension is three, we assume furthermore that it has a finite third moment, and in higher dimension we assume that it has some finite exponential moment. 
For $u\in \mathbf{\mathcal T}$ 
we let $\xi_u$ be its number of children, so that 
$$
\mathbb E[\xi_u]= 1,  \quad \mathbb V\text{ar}(\xi_u)= \sigma^2, 
\quad \text{for all }u \in \mathbf{\mathcal T}.
$$

We denote the root of the tree by $\emptyset$. We write $|u|$ the generation of a vertex $u\in \mathcal T$, i.e. its distance to the root of the tree. 
We let $u\wedge v$ be the least common ancestor of $u$ and $v$, i.e. the vertex at maximal distance from the root, among the ancestors of 
both $u$ and $v$.  For $n\ge 0$, we let $Z_n$ be the number of vertices at generation $n$; in particular by definition $Z_0= 1$ (the process $\{Z_n\}_{n\ge 0}$ is often called a BGW process, or sometimes just a Galton--Watson process in the literature). We also let $\mathcal T_n:=\{u\in \mathcal T:|u|\le n\}$. 
It follows from our hypotheses on $\mu$, that for any $n\ge 1$, one has 
\begin{equation}\label{Zn}
\mathbb E[Z_n] = 1, \quad \text{and}\quad \mathbb V\text{ar}(Z_n) = n\sigma^2.
\end{equation}
We also recall Kolmogorov's estimate (see \cite[Theorem 1 p.19]{AN}): 
\begin{equation}\label{Kol}
\mathbb P(Z_n\neq 0) \sim \frac{2}{\sigma^2 n }, \quad \text{as }n\to \infty.
\end{equation}

Recall the definition \eqref{def.Tr} of $\mathcal T(\Lambda)$, for $\Lambda\subset
\Z^d$, and for $n\ge 1$, set
\begin{equation}\label{ZnLambda}
\mathcal Z_n(\Lambda)=\{u\in \mathcal T(\Lambda):\ |u|=n\}. 
\end{equation}
We define the outer boundary of a subset $\Lambda\subset \mathbb Z^d$ as 
$$\partial \Lambda := \{ z \in \Lambda^c : \exists y\in  \Lambda \text{ with }\|y-z\|= 1\}. $$ 
We let $(X_e)_e$ be a collection of independent and identically distributed random variables indexed by the edges of the tree, with 
joint law the uniform measure on the neighbors of the origin in $\mathbb Z^d$ 
(for a formal construction, see for instance \cite{Shi}). Then we define the branching random walk $\{S_u,\ u\in \mathbf {\mathcal T}\}$, as the tree-indexed random walk, which means that for 
any vertex $u$, $S_u-S_\emptyset$ is the sum of the
random variables $X_e$ along the unique geodesic path from $u$ to the root. 
We write $\bE$ the expectation with respect to the BRW. We let  
$\mathbf P$ be the law of the standard random walk (which we shall also abbreviate as SRW) $\{S_n\}_{n\ge 0}$ on $\Z^d$, starting from the origin. 
For $x\in \mathbb Z^d$, we let $\mathbf P_x$ be the law of the SRW starting from $x$. For $r > 0$, 
we denote by $H_r$ the hitting time 
of $\partial B_r$, for the SRW: 
\begin{equation}\label{Hr}
H_r:= \inf\{n \ge 0 : S_n \in \partial B_r\}.
\end{equation}
If $d\ge 3$, we let 
$G$ be the Green's function, which is defined for any $z\in \mathbb Z^d$, by 
$$G(z) = \sum_{n=0}^\infty  \mathbf P(S_n= z). $$ 
We recall that under our assumption on the jump distribution, there exists a constant $c_G>0$ (only depending on the dimension), such that (see Theorem 4.3.1 in \cite{LL10}):  
\begin{equation}\label{Green}
G(z) = c_G\cdot \|z\|^{2-d} + \mathcal O(\|z\|^{-d}).
\end{equation}
Furthermore, the function $G$ is harmonic on $\mathbb Z^d\setminus \{0\}$, in the sense that for all $x$ different from the origin, 
$G(x) = \mathbf E_x[G(S_1)]$. As a consequence, using the optional stopping time theorem, we deduce that for some positive constants 
$c$ and $C$, one has for any $r\ge 1$ and any $x\notin B_r$, 
\begin{equation}\label{Green2}
\frac{cr^{d-2}}{\|x\|^{d-2}}\le \frac{G(x)}{\sup_{z\in \partial B_r} G(z)} 
\le \mathbf P_x(H_r<\infty) \le \frac{G(x)}{\inf_{z\in \partial B_r} G(z)}
\le \frac{Cr^{d-2}}{\|x\|^{d-2}}.
\end{equation}

As in \cite{AHJ}, we shall also make use of Paley-Zygmund's inequality, which asserts that for any nonnegative random variable $X$ having finite second moment, 
and for any $\varepsilon \in [0,1)$, 
\begin{equation}\label{PZ1}
\mathbb P(X\ge \varepsilon \mathbb E[X])  \ge \frac {(1-\varepsilon)^2\cdot \mathbb E[X]^2}{\mathbb E[X^2]}. 
\end{equation}
A usefull variant of \reff{PZ1}, which comes after a little
algebra reads 
\begin{equation}\label{PZ}
\mathbb P\big(X\ge \varepsilon\cdot \mathbb E[X\mid X\neq 0]\big)  \ge \frac {(1-\varepsilon)^2\cdot \mathbb E[X]^2}{\mathbb E[X^2]}. 
\end{equation}
Finally, given two functions $f$ and $g$, we write $f\lesssim g$, if there exists a constant $C>0$, such that $f\le Cg$, and similarly for $f\gtrsim g$. 


\section{Preliminary results}\label{sec-prelim}
\subsection{Exponential moments for the BGW process}
In this subsection, we prove two elementary facts on the BGW process. Recall that we assume the offspring distribution $\mu$ to have mean one, and some finite exponential moment. Our first result shows that some exponential moment of $Z_n/n$ is finite, conditionally on $Z_n$ being nonzero 
(which is also known to converge in law to an exponential random variable with mean one as $n$ goes to infinity, see \cite[Theorem 2 p.20]{AN}).  
\begin{lemma}\label{lem.BGW}
There exist $\lambda>0$, such that 
\begin{equation}\label{Zn.exp}
\sup_{n\ge 1} \ \mathbb E\big[\exp(\frac{\lambda Z_n}{n})\ \big|\ Z_n\neq 0\big] < \infty. 
\end{equation}
\end{lemma}
Note that the result is not new, and much stronger results are known, see for instance \cite{NV75,NV03}, but for reader's convenience 
we shall provide a direct and short proof here. 

\begin{proof}
Note that by \eqref{Kol}, there exists $c>0$, such that for all $n\ge 1$, 
\begin{align*}
\mathbb E\big[\exp(\frac{\lambda Z_n}{n})\ \big|\ Z_n\neq 0\big] & = 1+ \frac{\mathbb E\big[\exp(\frac{\lambda Z_n}{n})\big] - 1}{\mathbb P(Z_n\neq 0)}\\
& \le 1+ cn \left(\mathbb E\big[\exp(\frac{\lambda Z_n}{n})\big] - 1\right).  
\end{align*}
Thus it suffices to show that for some positive constants $c$ and $\lambda_0$, one has for all $\lambda\le \lambda_0$, and all $n\ge 1$,  
\begin{equation}\label{goal.Zn}
\varphi_n(\lambda):=\mathbb E\big[\exp(\frac{\lambda Z_n}{n})\big] \le \exp\big(\frac{\lambda+c\lambda^2}{n}\big). 
\end{equation}
We prove this by induction over $n$. Note that the result for $n=1$ follows from the fact that $Z_1$ is distributed as $\mu$, which 
has a finite exponential moment by hypothesis. 
Indeed, this implies that for some $c_0>0$, and all $\lambda$ small enough, 
\begin{align*}
\varphi_1(\lambda) = \mathbb E\big[\exp(\lambda Z_1)\big] & \le 1 + \lambda+ \lambda^2 \mathbb E\big[Z_1^2\exp(\lambda Z_1)\big] \le 1+  \lambda+ 
c_0 \lambda^2 \le \exp( \lambda+c_0 \lambda^2). 
\end{align*}
Now assume that~\eqref{goal.Zn} holds true for some $n$, and let us show it for $n+1$.   
Recall that conditionally on $Z_n$, $Z_{n+1}$ is distributed 
as a sum of $Z_n$ i.i.d. random variables with the same law as $Z_1$. 
Therefore, plugging the above computation yields 
$$
\varphi_{n+1}(\lambda)= \mathbb E\big[\exp(\frac{\lambda Z_{n+1}}{n+1})\big] = \mathbb E\left[ \varphi_1(\frac{\lambda}{n+1})^{Z_n}\right]\le  \varphi_n\Big(\frac{\lambda n}{n+1} + c_0\frac{\lambda^2 n}{(n+1)^2}\Big) . $$ 
Note that the induction hypothesis reads also as 
$n\log \varphi_n(\lambda)\le \lambda+ c\lambda^2$. Then 
\begin{equation}\label{wrong-1}
\begin{split}
(n+1)\log \varphi_{n+1}(\lambda) & \le 
\frac{n+1}{n}\Big(\frac{n}{n+1}\lambda+c_0\frac{n}{(n+1)^2}\lambda^2 +
c\big(\frac{\lambda n}{n+1} + c_0\frac{\lambda^2 n}{(n+1)^2}\big)^2\Big)\\
& \le \lambda +\frac{c_0}{n+1}\lambda^2+c\frac{n}{n+1}\lambda^2
+ 2cc_0\frac{\lambda^3}{n}+cc_0^2\frac{\lambda^4}{n^2}\\
& \le \lambda+ c\lambda^2-\frac{c-c_0}{n+1}\lambda^2+
2cc_0\frac{\lambda^3}{n}+cc_0^2\frac{\lambda^4}{n^2}. 
\end{split}
\end{equation}
Thus if we choose $c=2c_0$, and $\lambda$ small enough so that 
for any $n\ge 1$,
\[
4c_0 \frac{n+1}{n} \lambda+ 2c_0^2\frac{n+1}{n^2} \lambda^2<1,
\]
one obtains 
\begin{equation*}
\varphi_{n+1}(\lambda)\le \exp(\frac{\lambda +c\lambda^2}{n+1}). 
\end{equation*}
This establishes the induction step,  and concludes the proof of the lemma. 
\end{proof}
Our second result concerns the exponential moments for the total size of the BGW tree, up to some fixed generation, and is proved along the same lines. 
\begin{lemma}\label{lem.BGW2}
There exist positive constants $c$ and $\lambda_0$, such that for any $0\le \lambda\le  \lambda_0$, and any $n\ge 1$, 
\begin{equation*}
\mathbb E\big[\exp(\frac{\lambda |\mathcal T_n|}{n^2})\big] \le \exp(\frac{\lambda+c\lambda^2}{n}).   
\end{equation*}
\end{lemma}
\begin{proof}
We prove the result by induction on $n\ge 1$. The case $n=1$ has already been seen in the proof of Lemma~\ref{lem.BGW}, 
and only relies on the fact that $Z_1$ has a finite exponential moment by assumption. Assume now that it holds for some $n$, and let us prove 
it for $n+1$. Since conditionally on $Z_1$, $|\mathcal T_{n+1}|$ is the sum of $Z_1$ i.i.d. random variables distributed as $|\mathcal T_{n}|$, we deduce from the induction hypothesis, that for some $c$ and $\lambda_0$ one has for all $\lambda\le \lambda_0$, 
$$\mathbb E\Big[\exp(\frac{\lambda |\mathcal T_{n+1}|}{(n+1)^2})\ \big|\ Z_1\Big] 
\le \exp\big(\frac{\lambda\frac{n}{n+1}+c\lambda^2(\frac{n}{n+1})^3}{n+1}Z_1\big)\le \exp\Big(\frac{\lambda+c\lambda^2 - \frac{\lambda}{n+1} }{n+1}Z_1\Big) .$$
Integrating now both sides over $Z_1$, and using the result for $n=1$, gives 
$$\mathbb E\Big[\exp(\frac{\lambda |\mathcal T_{n+1}|}{(n+1)^2})\Big] \le \exp\Big(\frac{\lambda+c\lambda^2 - \frac{\lambda}{n+1}+\frac{c(\lambda + c\lambda^2)^2}{n+1} +\frac{c\lambda^2}{(n+1)^3}}{n+1}\Big) \le \exp\Big(\frac{\lambda+c\lambda^2 - \frac{\lambda - 2c\lambda^2+ \mathcal O(\lambda^3)}{n+1} }{n+1} \Big),$$
and the right-hand side is well smaller than $\exp(\frac{\lambda + c\lambda^2}{n+1})$, provided $\lambda_0$ is small enough. 
This concludes the proofs of the induction step, and of the lemma. 
\end{proof}


\subsection{Spine decomposition for the BRW}
We present here a spine decomposition of the BRW conditioned to hit a set, which was introduced by Zhu to derive upper bounds on hitting probabilities, see \cite{Zhu1, Zhu3}. We shall  
use it also later for proving the lower bound in Theorem~\ref{theo.main} 
in dimension four.

Following the terminology of \cite{Zhu1,Zhu3}, an
{\it adjoint BGW tree} is a BGW tree in which only the law 
of the number of children of 
the root has been modified, and follows the law $\tilde \mu$, given by $\tilde \mu(i) := \sum_{j\ge i+1} \mu(j)$, for $i\ge 0$. The associated 
tree-indexed random walk is the {\it adjoint BRW}. 
Then we define $k_\Lambda(x)$, for $x\in  \mathbb Z^d$, as the probability for an adjoint BRW starting from $x$ to hit $\Lambda$.  
Now, given an integer $n$ and 
a path $\gamma:\{0,\dots,n\} \to \mathbb Z^d$, 
we define, with $|\gamma|= n$,  
\begin{equation}\label{def.pLambda}
p_\Lambda(\gamma):= \prod_{i=0}^{|\gamma|-1}\theta\big(\gamma(i+1)-\gamma(i)\big) \cdot \big(1- k_\Lambda(\gamma(i))\big), 
\end{equation}
where we recall that $\theta$ is the uniform measure on the neighbors of the origin.
In other words, $p_\Lambda(\gamma)$ is the probability that a SRW starting from $\gamma(0)$ follows the path $\gamma$ during its first $n$ steps, 
when it is killed at each step with probability given by the function $k_\Lambda$ at its current position.

We next define the probability measures $\{\mu_\Lambda^z\}_{z\in \Lambda^c}$ on the integers by  
\begin{equation*}
\mu_\Lambda^z(m) := 
\sum_{\ell \ge 0} \mu(\ell + m +1) r_\Lambda(z)^\ell / \big(1-k_\Lambda(z)\big), \quad \text{for all }m\ge 0, 
\end{equation*}
where $r_\Lambda(z)$ is the probability that a BRW starting from $z$ does not visit $\Lambda$, conditionally on the root having only one child. 
We call biased BRW starting from $z$, a BRW starting from $z$, conditioned on the number of children of the root having law $\mu_\Lambda^z$.

Furthermore, a finite path $\gamma$, 
is said to go from $x$ to $\Lambda$, 
which we denote as $\gamma:x\to \Lambda$, 
if for $n=|\gamma|$, $\gamma(n)\in \Lambda$, 
and $\gamma(i)\notin \Lambda$ for all $i<n$. 
In other words this simply means that the path $\gamma$ is defined up 
to its hitting time of $\Lambda$ 
(note that it includes the possibility that $n=0$ and $\gamma(0) \in \Lambda$). We shall also later 
write for simplicity $\gamma:x\to y$, when $\Lambda$ is reduced to a single point $y$. Moreover, if $\Lambda\subset A$, we write $\gamma:A\to \Lambda$ when the path $\gamma$ is such that $\gamma(0)\in A$, and $\gamma$ goes from $\gamma(0)$ to $\Lambda$. 
Then given $x\in \mathbb Z^d$, and $\gamma:x\to \Lambda$, 
we call $\gamma$-biased BRW, the union of $\gamma$, together with for each $i\in \{0,\dots,|\gamma|-1\}$, 
a biased BRW starting from $\gamma(i)$, independently for each $i$, and starting from $\gamma(n)$ some usual independent BRW .

We are now in position to describe the law of the BRW starting
at some $x\in \mathbb Z^d$, and
conditioned on hitting $\Lambda$. 
For simplicity, we restrict ourselves to the part which intersects $\Lambda$, 
since this is the only one that is needed here. 
On the event $\{\ell_{\mathcal T}(\Lambda)>0\}$, 
one defines the first entry vertex, as the smallest vertex 
$u\in \mathcal T$ in the lexicographical order, for which $S_u\in \Lambda$. 
Then, denoting by $\overleftarrow{u}$  
the unique geodesic path in the BGW tree going from the root to 
the vertex $u$, we let $\Gamma=\Gamma(\mathcal T)$
be the path in $\mathbb Z^d$, 
made of the successive positions of the BRW along that path. 
Thus if $\overleftarrow{u} = (u_0=\emptyset, u_1,\dots,u_n=u)$, 
with $n=|u|$, then 
$$
\Gamma=\big(S_\emptyset=x,S_{u_1},\dots,S_{u_n}\big).
$$ 
Note that by definition $\Gamma$ 
is a path which goes from $x$ to $\Lambda$ in our terminology.  
The following result comes from \cite[Proposition 2.4]{Zhu3}. 

\begin{proposition}[\cite{Zhu3}] \label{prop.Zhu.spine}
Assume $d\ge 1$. Let $\Lambda\subset \Z^d$, 
and $x\in \mathbb Z^d$ be given. 
\begin{enumerate}
\item For any path $\gamma:x\to \Lambda$, one has 
\begin{equation}\label{zhu-1}
\mathbb P_x(\Gamma = \gamma, \ell_{\T}(\Lambda)>0)=p_\Lambda(\gamma).
\end{equation}
\item Furthermore, conditionally on $\{\Gamma=\gamma,\, 
\ell_{\T}(\Lambda)>0\}$, the trace of the BRW 
on $\Lambda$ has the same law as the trace of a $\gamma$-biased BRW. 
\end{enumerate}
\end{proposition} 
The product formula~\eqref{def.pLambda} 
defining $p_\Lambda$ implies that $\Gamma$ satisfies the (strong) 
Markov property, in the following sense.  
Given $x\in \mathbb Z^d$, we define a probability measure $\mathbb P_\Lambda^x$ on the set of paths $\gamma:x\to \Lambda$, by 
$$\mathbb P_\Lambda^x(\gamma ) := \frac{p_\Lambda(\gamma)}{\sum_{\gamma':x\to \Lambda} p_\Lambda(\gamma')},$$
which is nothing else than the law of $\Gamma$, 
conditionally on the event $\{\ell_{\mathcal T}(\Lambda)>0\}$. For convenience, we also set $\mathbb P_\Lambda^x(\gamma) = 0$, if $\gamma$ is not a path that goes from $x$ to $\Lambda$. 
Then we can state the Markov property as follows 
(we only state a particular case, but the same would hold 
for any stopping time). 
\begin{corollary}\label{cor.spine.Markov}
Let $\Lambda\subset A\subset \mathbb Z^d$, and $x\in \mathbb Z^d$ be given. Let $\tau_A:=\inf\{i\ge 0 :\Gamma(i) \in A\}$. 
Then for any path $\gamma:A\to \Lambda$, one has  
$$\mathbb P_\Lambda^x\Big((\Gamma(\tau_A), \dots) 
=\gamma\mid \Gamma(0),\dots,\Gamma(\tau_A) \Big) 
= \mathbb P_\Lambda^{\Gamma(\tau_A)}(\gamma).
$$
\end{corollary}
\begin{proof}
It suffices to notice that by~\eqref{def.pLambda}, for any path $\gamma_0:x\to \gamma(0)$,  
writing $\gamma\circ \gamma_0$ for the concatenation of $\gamma_0$ and $\gamma$, one has
\begin{align*}
 \mathbb P_\Lambda^x\Big((\Gamma(\tau_A),\dots)= \gamma\mid (\Gamma(0),\dots,\Gamma(\tau_A)) = \gamma_0 \Big) & = \frac{p_\Lambda(\gamma\circ \gamma_0)}{\sum_{\gamma':x\to \Lambda}p_\Lambda(\gamma')} \times \frac{\sum_{\gamma':x\to \Lambda} p_\Lambda(\gamma')}{p_\Lambda(\gamma_0)\sum_{\gamma'':\gamma(0)\to \Lambda} p_\Lambda(\gamma'')}\\
& = \frac{p_\Lambda(\gamma)\cdot p_\Lambda(\gamma_0)}{p_\Lambda(\gamma_0)\sum_{\gamma'':\gamma(0)\to \Lambda} p_\Lambda(\gamma'')} = \mathbb P_\Lambda^{\gamma(0)}(\gamma). 
\end{align*}
\end{proof}

\subsection{Hitting Probability Lower Bounds}\label{sec.hitting}
Here we derive rough lower bounds for the hitting probabilities of balls, using a second moment method. The result is as follows (recall \eqref{etar}): 

\begin{lemma}\label{lem.hitting}
There exists a constant $c>0$ (only depending on the dimension), 
such that for any $r\ge 1$ and $x\notin B_r$, 
\begin{equation} 
\mathbb P_x(|\eta_r|>0) \ge c\cdot \left\{
\begin{array}{ll}
\frac{r^{d-4}}{\|x\|^{d-2}} & \text{if }d\ge 5 \\
\frac 1{\|x\|^2\log (1+\frac{\|x\|}{r}) }& \text{if }d=4.
\end{array}
\right. 
\end{equation}
\end{lemma}
\begin{remark}\emph{Note that in dimension five and higher the result follows from \cite{Zhu1}, which proves as well an upper bound of the same order.
We include a short proof here, for the reader's convenience. 
In dimension four, a more precise asymptotic is proved in 
\cite{Zhu3} in case $r=1$ (see also \cite{Zhu4} for a rough upper bound still in the case $r=1$).} 
\end{remark}
\begin{remark}  \label{rem.lb.biased} \emph{We shall also use this lower bound in dimension four for a biased BRW. In this case the result follows immediately from the lemma, and the fact that a biased BRW has a probability bounded from below to have at least one child, from where starts a fresh usual BRW.}
\end{remark} 

\begin{proof}
Recall that we denote by $H_r$ the hitting time of $\partial B_r$ 
for a standard random walk.  
One has, using \eqref{Green2} for the last inequality, 
\begin{align*}
\mathbb E_x\big[|\eta_r|\big] & = \sum_{n=0}^\infty 
\mathbb E_x\Big[\sum_{|u|=n} \1\{u \in \eta_r \}\Big]=
 \sum_{n=0}^\infty \mathbb E[Z_n] \cdot \mathbf P_x(H_r = n) 
= \mathbf P_x(H_r<\infty ) \gtrsim  (\frac r{\|x\|})^{d-2}. 
\end{align*}
Now we bound the second moment. 
Recall that by definition, if $u\in \eta_r$, then none of its descendant can be in $\eta_r$. Thus summing first over all possible 
$w\in \mathbf{\mathcal T}$ 
and then integrating over all possible $u\neq v$, with $u\wedge v=w$ 
(and necessarily both $u$ and $v$ different from $w$),  gives (recall the notation \eqref{ZnLambda})
\begin{align*}
\mathbb E_x\big[|\eta_r|^2\big] & = 
\mathbf P_x(H_r<\infty) + \sum_{k=0}^\infty 
\mathbb E_x \Big[ \sum_{w\in \mathcal Z_k((\partial B_r)^c)} \xi_w(\xi_w-1) 
\mathbf P_{S_w}(H_r<\infty)^2  \Big] \\
& = \mathbf P_x(H_r<\infty) + \sigma^2 \sum_{k=0}^\infty 
\mathbb E_x\Big[\sum_{w\in \mathcal Z_k((\partial B_r)^c)}  \mathbf P_{S_w}(H_r<\infty)^2  \Big] \\
&\le \mathbf P_x(H_r<\infty) + \sigma^2 \sum_{k=0}^\infty 
\mathbb E_x \Big[ \sum_{|w| = k} \1\{S_w\notin \partial B_r\}\cdot
\mathbf P_{S_w}(H_r<\infty)^2  \Big]\\
& \lesssim (\frac{r}{\|x\|})^{d-2} + r^{2(d-2)}\sum_{k=0}^\infty 
\mathbf E_x\Big[\frac{\1\{S_k \notin \partial B_r\}}{\|S_k\|^{2(d-2)}}\Big] 
=  (\frac{r}{\|x\|})^{d-2} + r^{2(d-2)}\sum_{z\notin B_r} \frac{G(z-x)}{\|z\|^{2(d-2)}}\\
& \lesssim \left\{ 
\begin{array}{ll}
\frac{r^d}{\|x\|^{d-2} } & \text{if }d\ge 5 \\
\frac{r^4}{\|x\|^2} \cdot \log (1+\frac{\|x\|}{r}) & \text{if }d=4.
\end{array}
\right.
\end{align*}
The result follows using that by Cauchy-Schwarz inequality, 
$$\mathbb P_x(|\eta_r|>0) \ge \frac{\mathbb E_x\big[|\eta_r|\big]^2}{\mathbb E_x\big[|\eta_r|^2\big]}. $$ 
\end{proof}

The next result holds in any dimension. 
\begin{lemma}\label{lem.hit.inside}
Assume $d\ge 1$. There exists a constant $c>0$ (only depending on the dimension), such that for any $r\ge 1$,  
$$
\inf_{x\in B_r} \ \mathbb P_x(|\eta_r|>0) \ge c/r^2. $$ 
\end{lemma}
\begin{proof}
The proof is similar to the previous lemma. 
Note first that for any $x\in B_r$, 
$$\mathbb E_x[|\eta_r|] = \mathbf P_x(H_r<\infty) = 1,$$
and as before, 
$$\mathbb E_x\big[|\eta_r|^2\big]   \le 1 + \sigma^2\sum_{k=0}^\infty \mathbf P_x(H_r>k) = 1+ \sigma^2\, \mathbf E_x[H_r] \lesssim r^2. $$
The result follows. 
\end{proof}

Finally we state a result concerning the first and third moments 
of $|\eta_r|$, when starting from $\partial B_{2r}$.
\begin{lemma}\label{lem.mom.etar}
Assume $d\ge 1$. 
\begin{enumerate}
\item For any $x\in B_r^c$, 
one has $\mathbb E_x\big[|\eta_r|\big] = \mathbf P_x(H_r<\infty)$. 
In particular, when $d\ge 3$, there exists 
$\varepsilon>0$, such that for any $r\ge 1$, 
$$
\sup_{x\in \partial B_{2r}} \mathbb E_x\big[|\eta_r|\big] 
\le 1-\varepsilon.
$$
\item Assume $d\ge 4$. There exists $C>0$, such that for any $r\ge 1$, 
$$\sup_{x\in \partial B_{2r}} \mathbb E_x\big[|\eta_r|^3\big] \le Cr^4.$$
\end{enumerate}
\end{lemma}
\begin{remark}\emph{The last point of the lemma can be understood, by considering that starting from $\partial B_{2r}$, the probability to hit $\partial B_r$ is of order $1/r^2$, and conditionally on hitting it, the number of frozen particles $|\eta_r|$ is typically of order $r^2$. This reasoning would apply also 
to all other moments of fixed order, but the third moment will suffice for our purpose. }
\end{remark}
  
\begin{proof}
We start with the first point. The equality $\mathbb E_x\big[|\eta_r|\big]=\mathbf P_x(H_r<\infty)$ has already been seen in the proof of Lemma~\ref{lem.hitting}. Then the fact that if $\|x\|\ge 2r$, this quantity is bounded from above by some constant smaller than one, in dimension three and higher, follows from \eqref{Green} and \eqref{Green2}.

Let us prove the second point now. When we sum over triples, say $(u,v,w)\in \eta_r$, we distinguish two cases. 
Either at least two of them are equal, and we just bound the corresponding sum by three times the second moment of $|\eta_r|$, 
or the three points are distinct. In the latter case, we again 
distinguish between two possible situations: 
either, $u\wedge v = u\wedge w$, or $u\wedge v \neq u\wedge w$. 
In both cases, by summing first over $(u\wedge v)\wedge(u\wedge w)$ yields, 
for any $x\in \partial B_{2r}$, (recall the notation \eqref{ZnLambda}),
\begin{align}\label{moment3}
\nonumber \mathbb E_x\big[|\eta_r|^3\big]  & 
\le  3\mathbb E_x\big[|\eta_r|^2\big]   
+ \sum_{k=0}^\infty \mathbb E_x\Big[
\sum_{u\in \mathcal Z_k((\partial B_r)^c)}
\xi_u(\xi_u-1)(\xi_u-2) \mathbf P_{S_u}(H_r <\infty)^3\Big]\\
 &  + \sum_{k=0}^\infty 
\mathbb E_x\Big[\sum_{u\in  \mathcal Z_k((\partial B_r)^c)}
\xi_u(\xi_u-1) \mathbf P_{S_u}(H_r <\infty)\cdot \mathbb E_{S_u} \big[|\eta_r|^2\big] \Big]. 
\end{align}
Next, using \eqref{Green2} we get
\begin{align*} 
& \sum_{k=0}^\infty \mathbb E_x\Big[\sum_{u\in  \mathcal Z_k((\partial B_r)^c)}
\xi_u(\xi_u-1)(\xi_u-2) \mathbf P_{S_u}(H_r <\infty)^3\Big]\\
& \le C r^{3(d-2)} \sum_{k=0}^\infty 
\mathbf E_x\Big[\frac{\1\{S_k\notin B_r\}}{\|S_k\|^{3(d-2)}} \Big]\le Cr^{3(d-2)} \sum_{z\notin B_r} \frac{G(z-x)}{\|z\|^{3(d-2)}}\le Cr^2. 
\end{align*}
Finally for the last sum in ~\eqref{moment3}, we use the computation from the proof of Lemma~\ref{lem.hitting}, in particular the fact that 
the second moment of $|\eta_r|$ is $\mathcal O(r^2)$, when starting from $x\notin B_r$. 
This yields the upper bound 
$$Cr^{2d-2}\sum_{k=0}^\infty \mathbf E_x\Big[\frac{\1\{S_k\notin B_r\}}{ \|S_k\|^{2(d-2)}}\log(1+\frac{\|S_k\|}{r}) \Big] \le Cr^{2d-2}\sum_{z\notin B_r}
\frac{G(z-x)}{\|z\|^{2(d-2)}}\cdot \log (1+\frac{\|z\|}{r})\le Cr^4,$$
concluding the proof of the lemma.  
\end{proof}

\subsection{Hitting probability upper bounds}
In this section we provide upper bounds for hitting probabilities of spheres, which are of the same order as the lower bounds obtained previously.

The first result considers hitting probabilities of a sphere, starting from inside the ball. 
The result is not new, in particular it was already proved in~\cite[Proposition 10.3]{Zhu1} under only a first moment hypothesis on the offspring distribution (and mild condition on the jump distribution of the walk). For completeness, we provide here an alternative proof, which however requires a finite second moment of the offspring distribution. We also mention~\cite{Kes} which proves similar results.
\begin{proposition}[\cite{Zhu1}]\label{prop.hit.inside}
Assume $d\ge 1$. There exists $C>0$, such that for any $r\ge 1$, 
\begin{equation*}
\sup_{x\in B_{r/2}} \mathbb P_x(|\eta_r|>0) \le \frac{C}{r^2}.
\end{equation*}
\end{proposition}
\begin{proof}
Assume without loss of generality that $r^2$ is an integer. We first write for $x\in B_{r/2}$, using \eqref{Kol}, that for some constant $C>0$, 
\begin{align*} 
\mathbb P_x(|\eta_r| >0 ) & \le \sum_{k=0}^{r^2} \mathbb P_x(Z_k\neq 0, Z_{k+1}=0, |\eta_r|>0) + \mathbb P(Z_{r^2} \neq 0) \\
& \le\sum_{k=0}^{r^2} \mathbb P_x(Z_k\neq 0, Z_{k+1}=0, |\eta_r|>0) + \frac{C}{r^2}. 
\end{align*} 
Now fix some $k\le r^2$, and note that conditionally on $Z_k$, the probability for $Z_{k+1}$ to be zero is equal to $\mu(0)^{Z_k}$. Then by using a first moment bound, and the fact that for any vertex at generation $k$, the probability that the BRW reaches $\partial B_r$ along the line of its ancestors is given exactly by the probability for a SRW to reach $\partial B_r$, we get for some constant $c>0$, 
\begin{align} \label{ZkZk+1}
\nonumber \mathbb P_x(Z_k\neq 0, Z_{k+1}=0, |\eta_r|>0) & \le \mathbb E\big[Z_k \mu(0)^{Z_k}\big]\cdot \mathbf P_x(H_r\le k) \\
 & \lesssim \mathbb E\big[Z_k \mu(0)^{Z_k}\big]\cdot \exp(-cr^2/k), 
\end{align}
where the last inequality is well-known, see e.g. Proposition 2.1.2 in~\cite{LL10}. Letting $f_k(s) = \sum_{n=0}^\infty \mathbb P(Z_k=n)s^n$, the generating function of $Z_k$, one has 
$$
\mathbb E\big[Z_k \mu(0)^{Z_k}\big] = \sum_{n=0}^\infty n\mu(0)^n \mathbb P(Z_k=n) = \mu(0) f'_k(\mu(0)). 
$$
We claim that for some constant $C>0$ (depending only on $\mu$), one has 
\begin{equation}\label{exp.Zk}
f'_k(\mu(0)) \le \frac C{k^2}. 
\end{equation} 
Indeed, this follows from the results in~\cite{AN}. First, note that it suffices to prove~\eqref{exp.Zk} when $k$ is an even integer, since for any $s\le 1$, and $k\ge 1$, $f'_k(s) = f'(f_{k-1}(s)) f'_{k-1}(s)\le f'_{k-1}(s)$. 
Now the process $(Z_{2k})_{k\ge 0}$ is a critical branching process, whose offspring distribution $\mu'$ satisfies $\mu'(1) >0$. 
Hence Lemma 2 p.12 in~\cite{AN} shows that $\mathbb P(Z_k=n) \le \mathbb P(Z_k=1)\cdot \pi_n$, for all $n\ge 1$, and some constants $(\pi_n)_n$ in $[0,\infty]$. Then Theorem 2 p.13 shows that $\sum_{n\ge 1} n\pi_n\mu(0)^n <\infty$, and Corollary 2 p.23 gives that $\mathbb P(Z_k=1) \le C/k^2$, for some constant $C>0$, which altogether proves~\eqref{exp.Zk}. Injecting this estimate in~\eqref{ZkZk+1} and summing over $k$ concludes the proof of the lemma. 
\end{proof}

The second result concerns hitting probabilities of a ball starting from a point outside it. 

\begin{proposition}[\cite{Zhu1,Zhu4}] \label{prop.hit.upper}
There exists $C>0$, such that for any $r\ge 1$, and any $x\notin B_{2r}$, 
\begin{equation}
\mathbb P_x(|\eta_r|>0) \le C\cdot \left\{
\begin{array}{ll}
\frac{r^{d-4}}{\|x\|^{d-2}} & \text{if }d\ge 5 \\
\frac 1{\|x\|^2\log (\|x\|/r) }& \text{if }d=4.
\end{array}
\right. 
\end{equation}
\end{proposition}
\begin{proof}
The result in dimensions at least five is given in \cite{Zhu1}, and the proof in dimension four is identical to the case $r=1$, which is done in \cite{Zhu4}. 
We leave the details to the reader. 
\end{proof}

\begin{remark}\label{rem.hit.adjoint}
\emph{We shall also need the result of the proposition for an adjoint BRW. For this, one can use a union bound, by summing over all the children of the root. Since the expected degree of the root in an adjoint BGW tree is finite, under the hypothesis that the offspring distribution has finite second moment, the result for the adjoint BRW follows from Proposition~\ref{prop.hit.upper}.  }
\end{remark}


\subsection{Proof of Theorem~\ref{theo.main} for short times}

We deal in this section with  $t\le r^4$.

\paragraph{The upper bound.}
We use that $\ell_\T(B_r)$ coincides 
essentially with the total size of the BGW tree $\T$. 
Indeed, using~\eqref{Kol} and Markov's inequality, yields
\begin{align*}
\mathbb P(\ell_\T(B_r) > t) & 
\le \mathbb P(|\mathbf{\mathcal T}|>t)  \le  \mathbb P(Z_{\sqrt t}\neq 0) + \mathbb P(1+Z_1+\dots + Z_{\sqrt t} > t) \\
&\lesssim \frac 1{\sqrt t} + \frac 1t \cdot \mathbb E[Z_1+\dots+ Z_{\sqrt t}]  \lesssim \frac 1{\sqrt t}, 
\end{align*}
giving the desired upper bound. 

\paragraph{The lower bound.} 
We divide the proof in two steps. 
We first show that we can bring of the order of $\sqrt t$ 
particles in $B_{r/2}$ at a cost of order $1/\sqrt t$, 
and then we show that conditionally on this event, the union of the BRWs emanating from these particles are likely to spend a time $t$ in the ball $B_r$.

Let us start with the first step. Let $r_0 := t^{1/4}/2$. Note that by hypothesis $r_0 \le r/2$.  
By Proposition~\ref{prop.hit.inside}, one has for some constant $C>0$, 
$$\mathbb P(|\eta_{r_0}|> 0) \le C/\sqrt t,$$ 
and since $\mathbb E[|\eta_{r_0}|]=1$,
we have, for some constant $c>0$,  
$$\mathbb E\big[|\eta_{r_0}|\ \big|\ |\eta_{r_0}|> 0\big]\ge 
c\sqrt t. $$ 
Then, by using Paley-Zygmund's inequality \reff{PZ} (with $\epsilon=1/2$), 
and the second
moment estimates in the proof of Lemma~\ref{lem.hit.inside}, we deduce that for some constant $\rho_1>0$, 
\begin{equation}\label{lower.1st}
\mathbb P(|\eta_{r_0}|\ge \rho_1\sqrt t) \ge \frac{\rho_1}{\sqrt t},
\end{equation}
concluding the first step.

Now we move to the second step. 
Let $X:=\sum_{\sqrt t/4\le k\le {\sqrt t} } Z_k$.  
We first claim that for some $c_0>0$ (independent of $t$ and $r$) we have
\begin{equation}\label{lower.BGW1}
\mathbb P\big(X \ge  c_0 t\big) \ge\frac{c_0}{\sqrt t}.
\end{equation}  
Indeed, observe that $\mathbb E[X] \ge \frac{1}{2}\sqrt t$, and 
thus by Kolmogorov's estimate~\eqref{Kol}, one has for some $c>0$, 
\[
\mathbb E[X\mid X\neq 0] \ge ct.
\]
On the other hand, using \eqref{Zn}, we get 
$$
\mathbb E[X^2] \le \sqrt t\cdot 
\sum_{\sqrt t/4\le i\le \sqrt t}  
\mathbb E[Z_i^2] \le t^{3/2},
$$ 
and~\eqref{lower.BGW1} follows using Paley-Zygmund's inequality~\eqref{PZ}.

Now, define 
\begin{equation}\label{LB-4}
Y= \sum_{\sqrt t/4\le |u|\le \sqrt t} \1\{S_u \in B_r\}.
\end{equation}
One has for any $z\in \partial B_{r_0}$, 
$$
\mathbb E_z[Y\mid \mathbf {\mathcal T}] =  
\sum_{\sqrt t/4\le k\le\sqrt t} Z_k \cdot \mathbf P_z(S_k \in B_r) 
\ge \delta\cdot X, 
$$
with 
\begin{equation*}
\delta:=\inf_{x\in B_r}\inf_{k\le r^2} \mathbf P_x(S_k\in B_r)>0.
\end{equation*}
On the other hand, by definition one has 
$Y \le X$, and thus~\eqref{PZ1} 
gives 
$$
\mathbb P_z(Y \ge c_0\delta t/2 \mid \mathbf {\mathcal T}) 
\cdot \1\{X\ge c_0t\}
\ge \frac{\frac{1}{4}\mathbb E_z\big[ Y\mid \mathbf {\mathcal T}\big]^2}
{\mathbb E_z\big[Y^2\mid \mathbf {\mathcal T}\big]}\cdot \1\{X\ge c_0t \}
 \ge \frac{\delta^2}{4}\cdot \1\{X\ge c_0t\}. 
$$ 
Note also that by definition $Y\le \ell_{\mathcal T}(B_r)$. Hence, taking expectation on 
both sides of the last inequality and plugging~\eqref{lower.BGW1}, gives with $\rho_2:= c_0\delta^2/4$, 
\begin{equation}\label{LB-1}
 \inf_{z\in \partial B_{r_0}} \mathbb P_z(\ell_{\mathcal T}(B_r)\ge \rho_2 t) \ge \frac{\rho_2}{\sqrt t}.
 \end{equation}  
To conclude, notice that, as mentioned at the beginning of the proof, the time spent in $B_r$ is larger than the time spent by 
the union of the BRWs emanating from the particles in $\eta_{r_0}$ (which by definition are on $\partial B_{r_0}$). 
Hence, denoting by $\mathcal N$ a Binomial random variable with number of trials $\lfloor \rho_1\sqrt t\rfloor $ and probability of success $\frac{\rho_2}{\sqrt t}$, we get using  \eqref{lower.1st} together with \eqref{LB-1}, 
 for some constant $\rho>0$ (independent of $t$ and $r$), 
 $$\mathbb P(\ell_{\mathcal T}(B_r) \ge t) \ge\mathbb P(\ell_{\mathcal T}(B_r) \ge t, |\eta_{r_0}|\ge \rho_1\sqrt t) \ge \mathbb P(\mathcal N\ge 1/\rho_2)\cdot \mathbb P(|\eta_{r_0}|\ge \rho_1\sqrt t) \ge \frac{\rho}{\sqrt t}, $$ 
 as wanted.

 This concludes the proof of Theorem \ref{theo.main} in case $t\le r^4$. \hfill $\square$


\section{Proof of Theorem~\ref{theo.main} in low dimension}\label{sec-LD}
In dimension one, two and three, the proofs are easier, and only 
require small moment estimates, 
which enables us to use the results from \cite{AHJ}. 
More precisely, in dimension one and two both the lower and upper bounds only require estimates of the 
first and second moment, 
and in dimension three we need an additional third moment. We define for $d\in \{1,2,3\}$,
\begin{equation*} 
R = (t/r^d)^{\frac{1}{4-d}}, \quad \text{and}\quad N= R^2. 
\end{equation*}
Recall that one can assume here that $t\ge r^4$, in which case it amounts to show that 
$$\mathbb P(\ell_\T(B_r) > t) = \Theta (R^{-2}).$$
The proof is based on the following result, which is given by Lemma 4.3 in \cite{AHJ}. For $n\ge 0$, and $x\in \mathbb Z^d$, write 
$$\ell_n(x) = \sum_{|u|\le n} \1\{S_u= x\}.  $$ 
\begin{lemma}[\cite{AHJ}] \label{lem.AHJ}
If $\mu$ is critical and has a finite second moment, then 
\begin{equation*}
\sup_{x\in \mathbb Z^2} \mathbb E[\ell_n(x)^2] \lesssim n.  
\end{equation*} 
If additionally $\mu$ has a finite third moment, then 
$$\sup_{x\in \mathbb Z^3} \mathbb E[\ell_n(x)^3]\lesssim \sqrt n. $$
\end{lemma} 
As a consequence, writing $\ell_N(B_r):= \sum_{x\in B_r} \ell_N(x)$, we get   
\begin{eqnarray*}
\begin{array}{llll}
 \mathbb E[\ell_N(B_r)^2]  & \lesssim &  
r^4 N &   \text{if }d=2 \\
 \mathbb E[\ell_N(B_r)^3] & \lesssim & r^9 \sqrt{N} &  \text{if }d=3.
 \end{array}
\end{eqnarray*}
Note also that by linearity, one has when $d=1$,  
$$\mathbb E[\ell_N(B_r)] \lesssim r \sqrt N. $$ 
Therefore, for any $d\in \{1,2,3\}$, 
$$\mathbb P(\ell_\T(B_r) >t ) \le \mathbb P(Z_N \neq 0) + \mathbb P(\ell_N(B_r) >t )\le \frac 1N + \frac { \mathbb E[\ell_N(B_r)^d]}{t^d}\lesssim R^{-2}.$$

For the lower bounds we use Paley-Zygmund's inequality~\eqref{PZ}, 
which we apply with
$$
X:= \ell_{2MN}(B_r) - \ell_{MN}(B_r) = \sum_{k=MN+1}^{2MN}\,  
\sum_{|u|=k} \1\{S_u\in B_r\},
$$
where $M$ is some well chosen integer to be fixed later. 
We need the following first moment bounds: if $t\ge r^4$ (equivalently $N\ge r^2$), 
\begin{eqnarray*}
\mathbb E[X] & = & \sum_{k=MN+1}^{2MN}  \mathbb E\left[
\sum_{|u|=k} \1\{S_u \in B_r\}\right]  
 = \sum_{k=MN+1}^{2MN} \mathbb E[Z_k] \cdot \mathbf P(S_k \in B_r)  =   \sum_{k=MN+1}^{2MN}\mathbf P(S_k \in B_r) \\
 & \gtrsim & \left\{
 \begin{array}{ll}
 r\sqrt N & \text{if }d=1 \\
 r^2 & \text{if }d=2 \\
 r^3/\sqrt{N} & \text{if }d=3.
 \end{array}
 \right. 
 \end{eqnarray*} 
Note also that $\mathbb P(X>0) \le \mathbb P(Z_{MN}>0) \sim \frac2 {\sigma^2MN}$. Thus if $M$ is chosen large enough, in any dimension $d\in \{1,2,3\}$, for $t\ge r^4$, 
\begin{equation*}
\mathbb E[X\mid X>0] \geq t. 
\end{equation*} 
It follows that in dimensions one and two, 
$$\mathbb P(\ell_\T(B_r)\ge t ) \ge \mathbb P(X\ge t ) \gtrsim R^{-2}. $$ 
In dimension three we need to use a third moment asymptotic, due to the presence of a $\log $ term in the second moment. 
As noticed in~\cite[Lemma 4.4]{AHJ}, for a nonnegative random variable with a finite third moment: 
$$\mathbb P(X\ge \varepsilon \mathbb E[X\mid X>0] )\ge \frac{(1-\varepsilon)^{3/2}\mathbb E[X]^{3/2}}{\mathbb E[X^3]^{1/2}}.$$
Applying this with $X$ as above, 
we get as well $\mathbb P(\ell_\T(B_r)>t ) \gtrsim R^{-2}$, concluding the proof of Theorem~\ref{theo.main} in dimensions one, two and three.


\section{Proof of Proposition~\ref{prop.Tr}}\label{sec-Tr}
In this section, we deal with the exponential moments of 
$ |\mathcal T(B_r)|/r^4$. It amounts to show that there are positive constants $c$ and $\lambda_0$, such that for
any $\lambda<\lambda_0$, and any $r\ge 1$,
$$
\sup_{x\in B_r} \ \mathbb E_x\Big[
\exp\big(\lambda\frac{ |\mathbf{\mathcal T}(B_r)|}{r^4}\big)\Big] 
\le \exp(c\frac{\lambda}{r^2}).
$$
Assume without loss of generality that $r^2\in \mathbb N$. Recall the notation \eqref{ZnLambda} for $\mathcal Z_n(B_r)$, and let $Z_n(B_r) := |\mathcal Z_n(B_r)|$.   
Then for $0\le j<r^2$, we define 
\begin{equation*}
\Upsilon_j=\sum_{i=0}^\infty Z_{ir^2+j}(B_r).
\end{equation*}
Note that $ |\mathcal T(B_r)|=\sum_{j<r^2} \Upsilon_j$, and by Holder's
inequality
\begin{equation}\label{old-2}
\bE_x\Big[\exp(\lambda \frac{|\mathcal T(B_r)|}{r^4})\Big]=
\bE_x\Big[\prod_{j<r^2}\big(
\exp(\lambda\frac{\Upsilon_j}{r^2})\big)^{1/r^2}\Big]\le
\sup_{j<r^2} \bE_x\big[\exp(\lambda\frac{\Upsilon_j}{r^2})\big].
\end{equation}
Thus, we need a uniform exponential moment on the family
$\{\Upsilon_j/r^2\}_{j<r^2}$.
Recall \eqref{Hr} and note that for any $x\in B_r$, and $k\ge 0$, $\mathbb E_x[Z_k(B_r)] \le \mathbf P_x(H_{2r}\ge k)$. Furthermore, 
it is well-known that
there is $\rho<1$, such that
$$\sup_{r\ge 1} \sup_{x\in B_r} \mathbf P_x(H_{2r} \ge r^2)\le \rho.$$
It follows from these last two observations that for any $r\ge 1$, 
$$
\sup_{j<r^2}\sup_{x\in B_r} \mathbb E_x[Z_{j+r^2}(B_r)]  
\le  \sup_{x\in B_r} \mathbf P_x(H_{2r}\ge r^2) \le \rho.
$$
Let $\lambda_0$ be such that the conclusion of Lemma~\ref{lem.BGW} holds. 
Then, using that $Z_{j+r^2}(B_r)\le Z_{j+r^2}$, 
we get for some constant $c>0$ (that might change from line to line, and depend on $\lambda_0$), 
that for any $\lambda\le \lambda_0$, any $j<r^2$ and $x\in B_r$
\begin{equation*}
\begin{split}
\bE_x\big[\exp(\lambda \frac{Z_{j+r^2}(B_r)}{r^2})\big]& 
\le 1+ \frac{\rho \lambda}{r^2} + \lambda^2\mathbb E
\big[\frac{Z_{j+r^2}^2}{r^4}\exp(\lambda \frac{Z_{j+r^2}}{r^2})\big]\\
& \le 1+ \frac{\rho \lambda}{r^2} +
 c\lambda^2
\mathbb E\big[\exp(\lambda_0\frac{Z_{j+r^2}}{r^2})
\mid Z_{j+r^2}\neq 0\big]\cdot \mathbb P(Z_{j+r^2} \neq 0) \\
& \le 1+  \frac{\rho \lambda}{r^2} +\frac{c \lambda^2}{r^2},
\end{split}
\end{equation*}
where for the last inequality we used \eqref{Kol} and Lemma~\ref{lem.BGW}.
We deduce that there exists (some possibly smaller) $\lambda_0$, and $\gamma<1$, such that for 
all $0\le \lambda\le \lambda_0$, and all $r\ge 1$, 
\begin{equation}\label{old-4}
\sup_{x\in B_r}\sup_{j<r^2}
\bE_x\big[\exp(\lambda \frac{Z_{j+r^2}(B_r)}{r^2})\big]
\le  \exp(\frac{(1-\gamma)\lambda}{r^2}). 
\end{equation}
Now for $u\in \mathcal T$, and $n\ge 0$, we write $Z^u_n(B_r)$ for the random variable with the same law as $Z_n(B_r)$ but translated in the subtree emanating from vertex $u$. 
Note that for $i\ge 1$, 
\begin{equation*}
Z_{ir^2+j}(B_r)=\sum_{u\in \mathcal Z_{(i-1)r^2+j}(B_r) }
 Z^u_{r^2}(B_r). 
\end{equation*}
Then \reff{old-4} implies (after successive conditioning) that for $\lambda\le \lambda_0$, 
\begin{equation}\label{old-6}
\bE_x\Big[\exp\big( \frac{\lambda}{r^2}
\sum_{i=1}^\infty Z_{i r^2+j}(B_r)
-\frac{(1-\gamma)\lambda}{r^2} 
\sum_{i=1}^\infty  Z_{(i-1)r^2+j}(B_r) \big)\Big]\le 1.
\end{equation}
The next step is to choose $p$ and $q$ such that $q^2(1-\gamma)=1$ and $1/p+1/q=1$. 
By using \eqref{old-6} and applying Holder's inequality twice, we get that for $\lambda \le \frac{\lambda_0}{pq}$, 
\begin{equation*}
\begin{split}
\bE_x\big[\exp( \frac{\lambda}{r^2}\Upsilon_j)\big]& = \mathbb E_x\Big[\exp(
 \frac{\lambda}{r^2} Z_j(B_r))\cdot \exp\big(  \frac{\lambda}{r^2}
\sum_{i=1}^\infty Z_{ir^2+j}(B_r)\big)\Big]\\
& \le  \Big(\mathbb E_x\big[\exp(\frac{p\lambda}{r^2} Z_j)\big]\Big)^{1/p}
\Big(\mathbb E_x\big[\exp(\frac{q^2(1-\gamma)\lambda}{r^2}\Upsilon_j)\big]\Big)^{1/q^2}. 
\end{split}
\end{equation*}
Then \eqref{goal.Zn} gives for some constant $C>0$, and $\lambda$ small enough, 
\begin{equation*}
\bE_x\big[\exp( \frac{\lambda}{r^2}\Upsilon_j)\big]\le
\Big(\mathbb E_x\big[\exp(\frac{p\lambda}{r^2} Z_j)\big]\Big)^{\frac 1{\gamma p}}
\le e^{C \lambda/r^2}.
\end{equation*}
Now, using \reff{old-2}, the result follows. \hfill $\square$


\section{Proof of Theorem~\ref{theo.inside}}\label{sec-inside}
First note that as mentioned in the introduction, 
\eqref{inside1} follows from \eqref{inside2} and Lemma~\ref{lem.hit.inside}. 
Indeed, these imply that for some positive constants $c$ and $c'$, 
for all $x\in B_r$, and all $\lambda<\lambda_0$, 
\begin{equation*}
\begin{split}
\mathbb E_x\Big[\exp(\lambda \frac{|\eta_r|}{r^2})
\ \big| \ \eta_r\neq \emptyset\Big] 
=& 1+\frac{\mathbb E_x\Big[\exp(\lambda \frac{|\eta_r|}{r^2})\Big]
-1}{\mathbb P_x(|\eta_r|>0)}  \\
\le& 1+ cr^2 \Big(\exp(\frac{c\lambda}{r^2}) - 1\Big)  \le
1+c'\lambda \le \exp(c'\lambda).  
\end{split}
\end{equation*}
We now move to the proof of \eqref{inside2} and \eqref{inside3}.  
For $n\ge 0$, we use the notation $\mathcal T_n(\Lambda):=
\{u\in \mathcal T(\Lambda): |u|\le n\}$, and 
$\eta^n_r:=\{u\in \eta_r: |u|\le n\}$. Recall also the notation~\eqref{ZnLambda}. 
The main point is to observe that $\eta^n_r$ and $\mathcal T_n(B_r)$
are linked via some martingale $\{M_n\}_{n\ge 0}$, defined for $n\ge 0$, by 
\begin{equation}\label{def-Mar1}
M_n:= \sum_{u\in \mathcal Z_n(B_r)}  
\|S_u\|^2 +  \sum_{u\in \eta^n_r} \|S_u\|^2-
\Big(|\mathcal T_n(B_r)|+|\eta^n_r|\Big).
\end{equation}
The next lemma gathers the results needed about this process. 
\begin{lemma}\label{lem.M}
The following hold for the process $\{M_n, n\ge 0\}$: 
\begin{enumerate}
\item it is a martingale, with respect to the 
filtration $(\mathcal F_n)_{n\ge 0}$, where 
$\mathcal F_n = \sigma(\mathbf{\mathcal T}_n, \{S_u\}_{|u|\le n})$.  
\item Furthermore, it converges almost surely towards 
\[
M_\infty := \sum_{u\in \eta_r} \|S_u\|^2- 
\big(|\mathbf{\mathcal T}(B_r)|+
|\eta_r|\big). 
\] 
\end{enumerate}
\end{lemma}
\begin{proof}
The second part of the lemma is immediate  
since almost surely the tree $\mathcal T$ is finite, and thus $\mathcal Z_n(B_r)= \emptyset$, 
for all $n$ large enough. 
So let us prove the first point now. Set $\nabla M_n := M_{n+1}- M_n$, 
for $n\ge 0$, and for $u\in \mathbf{\mathcal T}$,  
let $\mathcal N(u)$ be the set of children of $u$. 
Recall that $\xi_u = |\mathcal N(u)|$ by definition.  
Then note that 
\begin{align*}
&  \sum_{u\in \mathcal Z_{n+1}(B_r)} \|S_u\|^2 = \sum_{u\in \mathcal Z_n(B_r)} \sum_{v\in \mathcal N(u)} \|S_v\|^2 - 
\sum_{u\in \eta_r^{n+1}\setminus \eta_r^n} \|S_u\|^2, \\
& |\mathcal T_{n+1}(B_r)| -|\mathcal T_n(B_r)|  = |\mathcal Z_{n+1}(B_r)|, \\
& \sum_{u\in \mathcal Z_n(B_r)} \xi_u = |\mathcal Z_{n+1}(B_r)| + |\eta_r^{n+1}| - |\eta_r^n|, \\
\end{align*}
which altogether yield 
\begin{equation}\label{nablaM}
\nabla M_n  = \sum_{u\in \mathcal Z_n(B_r)} \Big((\xi_u - 1) \|S_u\|^2 + \sum_{v\in \mathcal N(u)} ( \|S_v\|^2  - \|S_u\|^2-1)\Big).
\end{equation} 
The result follows since for $u \in \mathcal Z_n(B_r)$, 
$\xi_u$ is independent of $\mathcal F_n$, hence of $\|S_u\|$, and moreover, since the jump 
distribution of $S$ is centered and supported on the set of 
neighbors of the origin, one has for any $v\in \mathcal N(u)$, by Pythagoras, 
$$\mathbb E_x\Big[ \|S_v\|^2  \mid \mathcal F_n, \, \mathcal N(u)\Big] =  \|S_u\|^2 + 1. $$ 
\end{proof}

We can now finish the proof of the theorem, 
by showing \eqref{inside2} and \eqref{inside3}. 

\begin{proof}[Proof of \eqref{inside2}]
Observe first that when $1\le r\le 2$, the result follows from Proposition~\ref{prop.Tr}, since $\eta_r \subset \mathcal T(B_{r+1})$, for any $r\ge 1$. Hence one can assume now that $r\ge 2$. 
By definition, one has 
\[
\sum_{u\in \eta_r}\|S_u\|^2\ge r^2
|\eta_r|.
\]
Thus in view of Proposition~\ref{prop.Tr} and Lemma~\ref{lem.M}, 
it just amounts to show that $M_\infty/r^4$ has some finite exponential moment. To see this, first recall that by assumption the offspring distribution has some finite 
exponential moment. Therefore, for any $u\in \mathbf {\mathcal T}$, 
and $\lambda$ small enough, 
$$
\mathbb E \big[\exp(\lambda \frac{\xi_u - 1}{r^2})\big] \le 1+ c\frac{\lambda^2}{r^4} \le \exp(\frac{c\lambda^2}{r^4}),$$
for some constant $c>0$. Likewise, if $X$ has distribution $\theta$, then for any $z\in B_r$, by Cauchy-Schwarz, 
$$\big|\|z+X\|^2 - \|z\|^2 - 1\big| \le 2r \ \text{ and }\ \mathbb E\big[\|z+X\|^2\big]  =  \|z\|^2 + 1.$$
Therefore, there exists $c>0$, such that for every $\lambda\le 1$, 
$$\mathbb E \Big[\exp\big(\lambda\cdot  \frac{\|z+X\|^2 - \|z\|^2 - 1}{r^4}\big)\Big] \le 1+ c\frac{\lambda^2}{r^6} \le \exp(c\frac{\lambda^2}{r^6}).$$ 
It follows, using \eqref{nablaM} and bounding $\|S_u\|^2$ by $r^2$ in this formula, 
that for any $n\ge 0$, and $\lambda$ small enough, for some constant $c>0$, 
$$
\mathbb E_x\Big[\exp\big(\lambda \frac{\nabla M_n}{r^4}\big) 
\ \Big|\ \mathcal F_n\Big] \le \exp\big(c\lambda^2 
\frac{|\mathcal Z_n(B_r)|}{r^4}\big). 
$$ 
We deduce by successive conditioning, that for any $n\ge 1$, 
$$
\mathbb E_x\Big[\exp\Big(\frac{\lambda}{r^4}(M_n - M_0)
- c\lambda^2 \sum_{k= 0}^{n-1} \frac{|\mathcal Z_k(B_r)|}{r^4}\Big)\Big] \le 1.
$$
Note also that by definition for any $x\in B_r$, under $\mathbb P_x$,  
\[
M_0= (\|x\|^2-1) \le r^2,\quad
\text{ and }\quad \sum_{k\ge 0} |\mathcal Z_k(B_r)|\le |\mathcal T(B_r)|.
\] 
Hence, using Cauchy-Schwarz and Fatou's lemma, we obtain the existence of some positive constants $c$ and $\lambda_0$, such that for any $0\le \lambda\le \lambda_0$, any $r\ge 1$, and any $x\in B_r$, 
$$
\mathbb E_x\Big[\exp\big(\lambda\frac{M_\infty}{r^4}\big)\Big]
\le \liminf_{n\to \infty} \ \mathbb E_x\Big[\exp\big(\lambda\frac{M_n}{r^4}\big)\Big]
\le \mathbb E_x\Big[\exp\big(\frac{\lambda}{r^2}+
c\lambda^2\frac{|\mathcal T(B_r)|}{r^4}\big)\Big]. 
$$
Then \eqref{inside2} follows from Proposition~\ref{prop.Tr}. 
\end{proof}

\begin{proof}[Proof of \eqref{inside3}]
Note that for any $x\in B_r$, $\mathbb E_x\big[|\eta_r|\big] = 
\mathbf P_x(H_r<\infty) = 1$.
Therefore, by expanding the exponential, 
we find using Proposition~\ref{prop.hit.inside} and \eqref{inside2} at the third line, that for any $x\in B_{r/2}$, 
\begin{align*}
\mathbb E_x\Big[\exp(\lambda 
\frac{|\eta_r|}{r^2})\Big]  & \le 1 + \frac{\lambda}{r^2} + \frac{\lambda^2}{r^4} \mathbb E_x\Big[|\eta_r|^2\exp(\lambda \frac{|\eta_r|}{r^2})\Big]\\
&  = 1 + \frac{\lambda}{r^2} + \frac{\lambda^2}{r^4} \mathbb E_x\Big[|\eta_r|^2 \exp(\lambda \frac{|\eta_r|}{r^2})\ \Big|\ \eta_r \neq \emptyset \Big]\cdot \mathbb P_x(|\eta_r|>0) \\
&\le 1+ \frac{\lambda}{r^2} + c\frac{\lambda^2}{r^2} \mathbb E_x\Big[ \exp(\lambda_0 \frac{|\eta_r|}{r^2})\ \Big|\ \eta_r\neq \emptyset \Big] \le 1+ \frac{\lambda}{r^2} + c\frac{\lambda^2}{r^2}\\
&  \le \exp\big(\frac{\lambda + c\lambda^2}{r^2}\big), 
\end{align*}
where $\lambda_0$ is the constant appearing in the statement of \eqref{inside2}, and $c$ is another constant that might change from line to line (and depend on $\lambda_0$).  
\end{proof}


\section{Proof of Theorem~\ref{theo.outside5}}\label{sec-outside5}
\paragraph{Proof of \eqref{outside1}.}
We assume here that $d\ge 5$. Let $r\ge 1$ be given and $x$ satisfying $\|x\|\ge 2r$.  
We recall that $\eta_{r,R}=\eta_r\cap \mathcal T(B_R)$. We also define $\eta_{r,R}^u$, for $u\in \mathcal T$, as the random variable with the 
same law as $\eta_{r,R}$, but in the subtree emanating from $u$, and similarly for other variables with an additional upper script $u$.

Let $i_0$ be the smallest integer such that $\|x\|\le r2^{i_0}$.
Define $R_0 = r2^{i_0}$, $Z_{r,0}=|\eta_{r,R_0}|$. Let also for $i\ge 1$, 
\begin{equation}\label{def-R}
R_i := 2^{i}R_0,\quad\text{and}\quad 
Z_{r,i}:= \sum_{u\in \eta_{R_{i-1}} } |\eta^u_{r,R_i}|.
\end{equation}
Then by definition, under $\mathbb P_x$, 
\begin{equation*}
|\eta_r| \le \sum_{i\ge 0} Z_{r,i}.
\end{equation*}
Thus, by monotone convergence, one has for any $\lambda\ge 0$, $r\ge 1$, 
and $\|x\|\ge 2r$ 
\begin{equation}\label{link.etar.etari}
\mathbb E_x \Big[\exp\big(\lambda\frac{|\eta_r|}{r^2}\big)\Big]
=\lim_{I\to \infty} \ \mathbb E_x\Big[\exp\big(\lambda\frac{\sum_{i=0}^I 
Z_{r,i}}{r^2}\big)\Big]. 
\end{equation}
Furthermore, if for $i\ge 0$, we let $\mathcal G_i$ denote the sigma-field 
generated by the BRW frozen on $\partial B_{R_i}$, then 
on one hand, $Z_{r,j}$ is $\mathcal G_i$-measurable, for all $j\le i$, 
and conditionally on $\mathcal G_{i-1}$, $Z_{r,i}$ 
is a sum of $|\eta_{R_{i-1}}|$ 
independent random variables. 
It remains now to bound their exponential moment. 

\begin{proposition}\label{prop.phi.i}
Assume $d\ge 5$. For $r\ge 1$, $R\ge r$, and $\lambda\ge 0$, let
\begin{equation*}
\varphi_{r,R}(\lambda) := \sup_{R\le \|x\|\le 2R} 
\mathbb E_x\Big[\exp\big(\lambda\frac{|\eta_{r,2R}|}{r^2}\big)\Big].
\end{equation*}
There exist positive constants $c$, $r_0$ and $\lambda_0$, such that for all $r\ge r_0$, $R\ge r$, and $0\le \lambda\le \lambda_0$, 
\begin{equation}\label{phi.rR}
\varphi_{r,R}(\lambda) \le \exp\big(\frac{c\lambda}{R^2}\big). 
\end{equation}
\end{proposition}
We postpone the proof of this proposition to the end of this section, and continue the proof of \eqref{outside1}. 
Note that the proposition implies with our previous notation, assuming $r\ge r_0$ and $\lambda\le \lambda_0$,
\begin{equation*}
 \mathbb E_x\Big[\exp\Big(\lambda\frac{Z_{r,0}}{r^2}\Big)\Big] \le  \exp\big(\frac{c\lambda}{\|x\|^2}\big).
\end{equation*}
Furthermore, by combining Proposition~\ref{prop.phi.i} with Theorem~\ref{theo.inside}, we deduce that for all $i\ge 1$, 
almost surely 
\begin{equation*}
 \mathbb E_x\Big[\exp\Big(\lambda\frac{Z_{r,i}}{r^2} 
- c\lambda\frac{ |\eta_{R_{i-1}}|}{R_{i-1}^2}\Big)\ \big|\ \mathcal G_{i-1}\Big] \le 1.
\end{equation*}
It follows by induction that for any $I\ge 1$, and $\lambda\le \lambda_0$, 
$$
\mathbb E_x\Big[\exp\Big(\lambda\sum_{i=1}^I \frac{ Z_{r,i}}{r^2}
-c\lambda \sum_{i=0}^{I-1}\frac{ |\eta_{R_i}|}{R_i^2}\Big)
 \ \big|\ \mathcal G_0
\Big] \le 1.$$
Using next Cauchy-Schwarz inequality, we get that for any $I\ge 1$, and $\lambda\le \lambda_0/2$, 
$$\mathbb E_x\Big[\exp\big(\lambda \sum_{i=0}^I 
\frac{ Z_{r,i}}{r^2} \big)\Big] 
\le  \exp\big(\frac{c\lambda}{\|x\|^2}\big) \cdot 
 \mathbb E_x\Big[\exp\big(2c\lambda \sum_{i=0}^{I-1}
\frac{ |\eta_{R_i}|}{R_i^2}\big)\Big]^{1/2}. $$ 
Now in order to compute the exponential moment in the right-hand side, 
we use Theorem \ref{theo.inside}. Indeed, note that for any $i\ge 1$, 
\begin{equation*}
\eta_{R_{i+1}} = \bigcup_{u\in \eta_{R_i}} 
\eta^u_{R_{i+1}}.
\end{equation*} 
Since as we condition on $ \eta_{R_i}$ the subtrees emanating from the vertices in 
$\eta_{R_i}$ are independent,
\eqref{inside3} shows that for $\lambda$ small enough, for any $i\ge 0$, 
\begin{equation*}
\mathbb E_x\Big[\exp\big(\lambda \frac{ |\eta_{R_{i+1}}|}{R_{i+1}^2}\big)
\ \Big|\ \eta_{R_i}\Big] 
\le \exp\big(2\lambda \frac{ |\eta_{R_i}|}{R_{i+1}^2}\big)
= \exp\big(\lambda \frac{ |\eta_{R_i}|}{2R_i^2}\big). 
\end{equation*} 
Since $R_0\ge \|x\|$, it follows by induction, that 
for any $I\ge 1$, and all $\lambda$ small enough, 
\begin{equation*}
\mathbb E_x\Big[\exp\big(\lambda \sum_{i=0}^{I-1}
\frac{ |\eta_{R_i}|}{R_i^2}\big)\Big]
\le \exp\big(\frac{2\lambda(1+\dots+2^{-I})}{\|x\|^2}\big) 
\le \exp(\frac{4\lambda}{\|x\|^2}).  
\end{equation*}
Together with \eqref{link.etar.etari}, this concludes the proof of \eqref{outside1}.

Then the proofs of \eqref{outside2} and \eqref{outside3} follow exactly as for the corresponding estimates, respectively \eqref{inside1} and \eqref{inside3}, from Theorem \ref{theo.inside}. For \eqref{outside3}, it suffices to use in addition the first point of Lemma~\ref{lem.mom.etar}. 
\hfill $\square$

\paragraph{Proof of Proposition~\ref{prop.phi.i}.}
For this we need two preliminary results, 
Lemmas~\ref{lem.phir2} and~\ref{lem.phir3} below. 
We note that the first one holds in fact in any dimension $d\ge 3$, 
and will be used also for the case $d=4$, in the next section. Recall two handy
notation: if $r<R$, we write
\[
\eta_{r,R}=\eta_r\cap \mathcal T(B_R),\quad\text{and}\quad
\eta_{R,r}=\eta_R\cap \mathcal T((\partial B_r)^c).
\]
In other words, $\eta_{R,r}$ is the set of particles which freeze on $\partial B_R$ before reaching $\partial B_r$. 
\begin{lemma}\label{lem.phir2} Assume $d\ge 3$.  
Define for $r\ge 1$, and $\lambda>0$, 
$$
\varphi_r(\lambda):=  \sup_{r\le \|x\|\le 2r} 
\mathbb E_x \Big[\exp\big(\lambda\frac{|\eta_{r,2r}|}{r^2}\big)\Big].
$$
There exist positive constants $c$ and $\lambda_0$ 
(only depending on the dimension), 
such that for any $r\ge 1$, and $0\le \lambda\le \lambda_0$,  
$$\varphi_r(\lambda) \le \exp\big(\frac{c\lambda}{r^2}\big). $$ 
\end{lemma}
\begin{proof}
Consider the process $\{\widetilde M_n\}_{n\ge 0}$ 
defined for $n\ge 0$ by 
\begin{equation*}
\widetilde M_n := \sum_{u\in \mathcal Z_n(B_{2r}\bs \partial B_r)} G(S_u) +  
\sum_{\substack{u\in \eta_{r,2r}\\ |u|\le n}}G(S_u)+
\sum_{\substack{u\in \eta_{2r,r}\\ |u|\le n}} G(S_u).
\end{equation*}
Note that for each $n\ge 0$, one has 
$$\nabla \widetilde M_n:= \widetilde M_{n+1}- \widetilde M_n = 
\sum_{u\in \mathcal Z_n(B_{2r}\bs \partial B_r)} 
\Big\{(\xi_u-1)G(S_u) + \sum_{v\in \mathcal N(u)}(G(S_v) - G(S_u))\Big\},$$ 
where we recall that $\mathcal N(u)$ denotes the set of children of $u$. 
Therefore, since $G$ is harmonic on $\mathbb Z^d\setminus \{0\}$, this process is a martingale with respect to the filtration 
$\{\mathcal F_n\}_{n\ge 0}$, as defined in Lemma~\ref{lem.M}.
Moreover, as $n\to \infty$, it converges almost surely toward $\widetilde M_\infty$ given by 
\begin{equation*}
\widetilde M_\infty = \sum_{u\in \eta_{r,2r}} G(S_u) + \sum_{u\in \eta_{2r,r}} G(S_u).
\end{equation*}
Letting $G(r):= \inf_{x\in \partial B_r} G(x)$, we thus have 
$\widetilde M_\infty \ge G(r) |\eta_{r,2r}|$.  
By Fatou's lemma, this yields 
\begin{equation}\label{phi.r.1}
\varphi_r(\lambda) \le \liminf_{n\to \infty}\sup_{r\le \|x\|\le 2r} 
\mathbb E_x\Big[\exp\big(\frac{\lambda \widetilde M_n}{r^2G(r)}\big)\Big]. 
\end{equation}
Now, as in the proof of~\eqref{inside2} one has for some constant $c>0$, for any $n\ge 0$, and any $\lambda$ small enough, 
$$
\mathbb E_x\Big[\exp\big(\frac{\lambda (\widetilde M_n-\widetilde M_0)}
{r^2G(r)} - \frac{c\lambda^2|\mathcal T_n(B_{2r})|}{r^4}\big)\Big] \le 1,
$$
from which we infer using Cauchy-Schwarz and the fact that under $\mathbb P_x$, 
$\widetilde M_0=G(x)$,  
\begin{equation*}
\mathbb E_x\Big[\exp\big(\frac{\lambda \widetilde M_n}{r^2G(r)}\big)\Big] 
\le \mathbb E_x\Big[\exp\big(\frac{2c\lambda^2
|\mathcal T_n(B_{2r})|}{r^4}\big)\Big]^{1/2}
\cdot \exp\big(\frac{\lambda G(x)}{r^2G(r)}\big).
\end{equation*}
Thus the lemma follows from Proposition~\ref{prop.Tr} and~\eqref{Green}, together with~\eqref{phi.r.1}. 
\end{proof}
We can now state the following result, 
which will be our main building block in the proof of Proposition~\ref{prop.phi.i}. 

\begin{lemma}\label{lem.phir3}
Assume $d\ge 5$. 
There exist positive constants $r_0\ge 1$ and $\lambda_0$, such that for any $r\ge r_0$, and $0\le \lambda\le \lambda_0$,  
$$\sup_{x\in \partial B_{2r}} \, \mathbb E_x \Big[\exp\big(\frac{\lambda |\eta_{r,4r}|}{r^2} + \frac{\lambda |\eta_{4r,r}|}{16r^2}\big)\Big] 
 \le \exp\Big(\frac{\lambda}{4r^2}\Big). $$ 
\end{lemma}
\begin{proof}
We note that for any $x\in \partial B_{2r}$, 
$$\mathbb E_x\Big[\exp\big(\frac{\lambda |\eta_{r,4r}|}{r^2}\big)\Big] \le 1+ \frac{\lambda \mathbb E_x\big[|\eta_{r,4r}|\big]}{r^2} + 
\frac{\lambda^2}{r^4} \mathbb E_x\Big[|\eta_{r,4r}|^2\exp\big(\frac{\lambda |\eta_{r,4r}|}{r^2}\big)\Big]. $$ 
Now, similarly as in Lemma~\ref{lem.mom.etar}, one has 
$$\mathbb E_x\big[|\eta_{r,4r}|\big] = \mathbf P_x(H_r<H_{4r}) \le  \frac{G(x) - G(4r)}{G(r)- G(4r)},$$
with $G(s)= \inf_{z\in \partial B_s} G(z)$, for $s\ge 1$. Therefore, using \eqref{Green}, we deduce that for $r$ large enough, 
$$\sup_{x\in \partial B_{2r}} \mathbb E_x\big[|\eta_{r,4r}|\big] \le \frac{1}{8} .$$ 
Using next Proposition~\ref{prop.hit.upper}, and (the proof of) Lemma~\ref{lem.phir2} we deduce, as for the proof of ~\eqref{inside3}, that for $\lambda$ small enough, 
$$\sup_{x\in \partial B_{2r}} \mathbb E_x\Big[|\eta_{r,4r}|^2\exp\big(\frac{\lambda |\eta_{r,4r}|}{r^2}\big)\Big] \le c r^2,$$
for some constant $c>0$. It follows that for $\lambda$ small enough, and $r$ large enough, 
 $$\sup_{x\in \partial B_{2r}} \mathbb E_x\Big[\exp\big(\frac{\lambda |\eta_{r,4r}|}{r^2}\big)\Big] \le \exp\big(\frac{\lambda}{6r^2}\big). $$ 
 The same argument leads to 
 $$\sup_{x\in \partial B_{2r}} \mathbb E_x\Big[\exp\big(\frac{\lambda |\eta_{4r,r}|}{16r^2}\big)\Big] \le \exp\big(\frac{\lambda}{16r^2}\big),$$
 and the lemma follows by using Cauchy-Schwarz inequality, since $\tfrac{1}{6} + \tfrac{1}{16} \le \tfrac 14$.  
\end{proof}

\begin{proof}[Proof of Proposition~\ref{prop.phi.i}] 
Assume $r\ge r_0$, with $r_0$ given by Lemma~\ref{lem.phir3}, and for $i\ge 0$, set $R_i = r2^i$. Let also 
$$
\phi_{r,i}(\lambda) :=  \sup_{x\in \partial B_{R_i}} 
\mathbb E_x\Big[\exp\big(\frac{\lambda|\eta_{r,R_{i+1}}|}{r^2} 
+ \frac{\lambda|\eta_{R_{i+1},r}|}{R_{i+1}^2}\big)\Big].  
$$
We will prove by induction that, 
for all $r\ge r_0$, and all $0\le \lambda\le \lambda_0$, with $\lambda_0$ as in Lemma~\ref{lem.phir3}, one has for all $i\ge 0$, 
\begin{equation}\label{goal.phiri}
\varphi_{r,i}(\lambda)\le \exp\big(\frac{\lambda}{R_i^2}\big). 
\end{equation} 
We claim that this implies the proposition. 
Indeed, let $R\ge r\ge r_0$ be given and assume that $R_i\le R< R_{i+1}$, for some $i\ge 0$. Let also $x$ be such that $R\le \|x\|\le 2R$. 
If $R\le \|x\|\le R_{i+1}$, then under $\mathbb P_x$, 
$$
|\eta_{r,2R}|\le \sum_{u\in \eta_{R_i,R_{i+1}}} |\eta^u_{r,R_{i+1}}| 
+ \sum_{u\in \eta_{R_{i+1},r}}   |\eta^u_{r,R_{i+2}}|,
$$
showing that the desired result follows indeed from~\eqref{goal.phiri}, 
Lemma~\ref{lem.phir2}, and~\eqref{inside2}.  
On the other hand, if $R_{i+1}\le \|x\|\le 2R$, then under $\mathbb P_x$, 
$$
|\eta_{r,2R}|\le \sum_{u\in \eta_{R_{i+1},R_{i+2}} } 
|\eta^u_{r,R_{i+2}}| + \sum_{u\in \eta_{R_{i+2},r} }
|\eta^u_{r,R_{i+3}}|,
$$
from which the result follows as well. 

Thus it only amounts to prove~\eqref{goal.phiri}, 
which we now show by induction on $i\ge 0$. 

Note that when $i=0$, the result is immediate by definition, and when $i=1$, the result is given by Lemma~\ref{lem.phir3}.

Assume next that it holds up to some integer $i\ge 1$,  
and let us prove it for $i+1$. 
For $x\in \partial B_{R_{i+1}}$, we define inductively four sequences 
$\{\zeta_k^j\}_{k\ge 0}$, for $j\in \{0,1,2,3\}$,  
of vertices of $\mathcal T$ as follows. 
Let $\zeta_0^2: = \{\emptyset\}$ be the root of $\mathcal T$. 
Next, we can first define for any $k\ge 0$, 
$$
\zeta_k^1 = \bigcup_{u\in \zeta_k^2} 
\eta^u_{R_i, R_{i+2}}, 
\quad\text{and}\quad \zeta_{k+1}^2 = \bigcup_{u\in \zeta_k^1} 
\eta^u_{R_{i+1},r}.
$$
Then we let 
$$
\zeta_k^0 := \bigcup_{u \in\zeta_k^1} \eta^u_{r,R_{i+1}},
\quad\text{and}\quad \zeta_k^3 = \bigcup_{u\in \zeta_k^2} 
\eta^u_{R_{i+2},R_i}. 
$$ 
In particular under $\mathbb P_x$, with $x\in \partial B_{R_{i+1}}$, one has for any $k\ge 0$, 
\begin{equation}\label{positions}
S_u \in \partial B_{R_{i+j-1}}, \text{ if }u\in \zeta_k^j, \text{ for }  j\in\{1,2,3\}, \quad \text{and}\quad S_u \in \partial B_r, \text{ if }u\in \zeta_k^0. 
\end{equation}
See Figure~\ref{fig:dessin-d5} below where 
an illustration of $\{\zeta_k^0\}$ is drawn. Moreover,  
$$
\eta_{r,R_{i+2}} = \bigcup_{k=0}^\infty \zeta_k^0, \quad\text{and}\quad 
\eta_{R_{i+2},r} = \bigcup_{k=0}^\infty \zeta_k^3.$$
Indeed, concerning the first equality, note that 
any particle reaching $\partial B_r$, before hitting $\partial B_{R_{i+2}}$, will make a number of excursions between 
$\partial B_{R_i}$ and $\partial B_{R_{i+1}}$ back and forth, before at some point reaching $\partial B_{R_i}$, 
and then $\partial B_r$ without hitting $\partial B_{R_{i+1}}$, 
and a similar argument leads to the second equality. 

\begin{figure}[htpb]
\centering
\includegraphics[width=8cm,height=10cm]{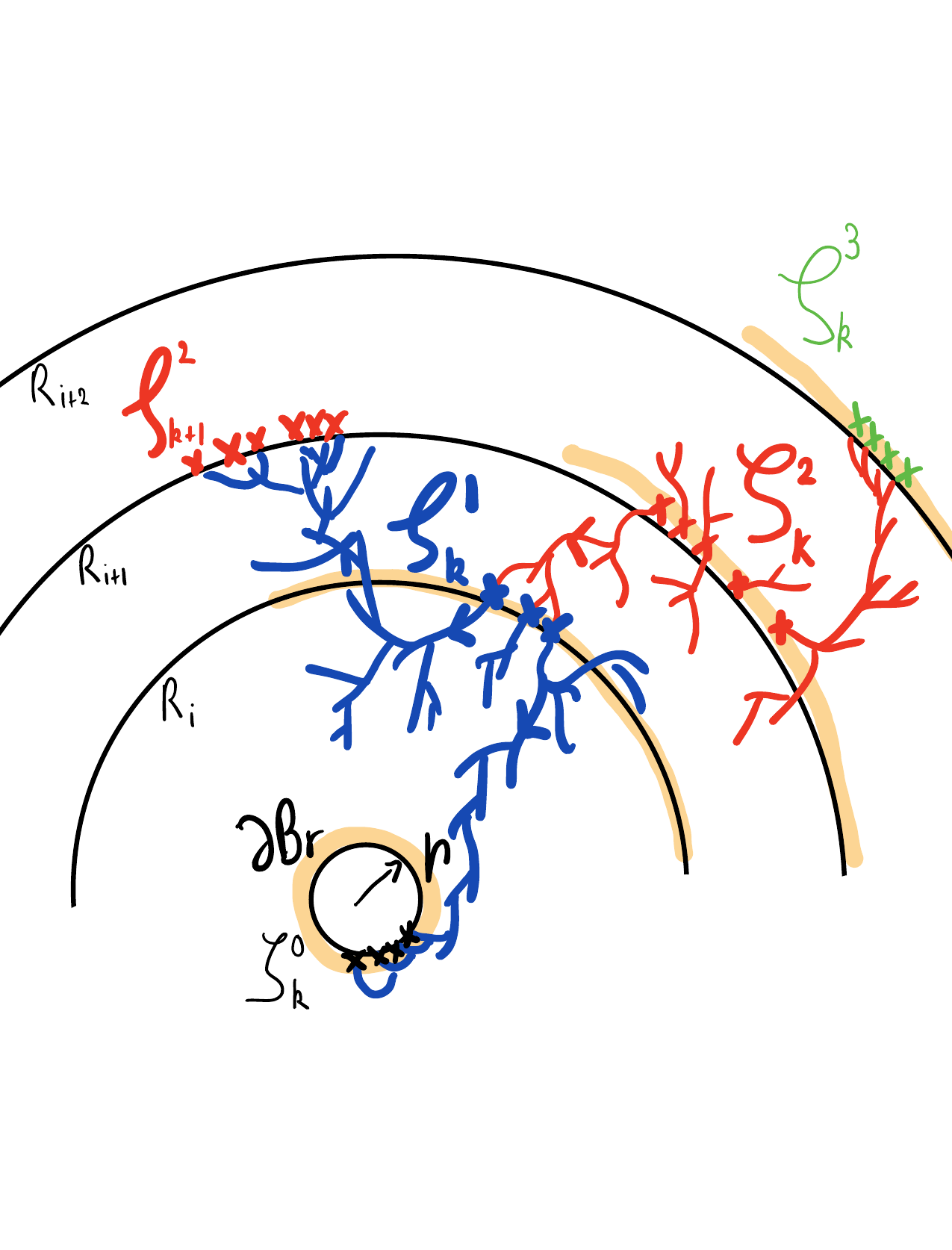}
\caption{Waves}\label{fig:dessin-d5}
\end{figure}

By monotone convergence, we deduce that  
\begin{equation}\label{eta.conv.mon}
\varphi_{r,i+1}(\lambda) \le \lim_{n \to \infty} \sup_{x\in \partial B_{R_{i+1}}}
\mathbb E_x\Big[\exp\big(\frac{\lambda\sum_{k=0}^n |\zeta_k^0|}{r^2}+\frac{\lambda\sum_{k=0}^n |\zeta^3_k|}{R_{i+2}^2}\big)\Big]. 
\end{equation} 
For $k\ge 0$, we let $\mathcal G_k$ be the sigma-field generated 
by the tree $\mathcal T$ cut at vertices in 
$\zeta_k^1\cup \zeta_k^3$, together with the positions of the BRW at the vertices on this subtree. 
We also let $\mathcal H_k$ be the sigma-field generated 
by the tree cut at vertices in  $\zeta_k^2$, together with the positions of the BRW on the corresponding subtree.

The induction hypothesis implies that almost surely, one has 
$$\mathbb E_x\Big[\exp\big(\frac{\lambda |\zeta_n^0|}{r^2}\big) \ \big|\ \mathcal G_n \Big] \le \exp\big(\frac{\lambda|\zeta_n^1|}{R_i^2}\big). $$ 
Then Lemma~\ref{lem.phir3} ensures,  
that for $0\le \lambda \le \lambda_0$, and $r\ge r_0$, almost surely (recall~\eqref{positions}), 
$$\mathbb E_x\Big[\exp\big(\frac{\lambda |\zeta_n^1|}{R_i^2}+ \frac{\lambda |\zeta^3_n|}{R_{i+2}^2}\big) \ \big|\ \mathcal H_n \Big] \le 
\exp\big(\frac{\lambda |\zeta_n^2|}{R_{i+1}^2}\big). $$  
Applying again the induction hypothesis, we get that almost surely, 
$$\mathbb E_x\Big[\exp\big(\frac{\lambda |\zeta_{n-1}^0|}{r^2}+ \frac{\lambda |\zeta^2_n|}{R_{i+1}^2}\big) \ \big|\ \mathcal G_{n-1} \Big] 
\le \exp \big(\frac{\lambda |\zeta_{n-1}^1|}{R_i^2}\big). $$  
Then an elementary induction shows that for all $n\ge 1$, 
$$\mathbb E_x\Big[\exp\big(\frac{\lambda\sum_{k=0}^n |\zeta_k^0|}{r^2}+\frac{\lambda\sum_{k=0}^n |\zeta^3_k|}{R_{i+2}^2}\big)\Big] 
\le \exp\big(\frac{\lambda}{R_{i+1}^2}\big),$$
proving \eqref{goal.phiri} for $i+1$, 
which concludes the proof of Proposition~\ref{prop.phi.i}.   
\end{proof}


\section{Proof of Theorem~\ref{theo.outside4}}\label{sec-outside4}
The proof is similar to the proof of Theorem~\ref{theo.outside5}, 
but one has to be slightly more careful. 
The main difference comes from the following modified version of Proposition~\ref{prop.phi.i}. 

\begin{proposition}\label{prop.phi.d4}
Assume $d=4$. There exist positive constants $c$, $r_0$ and $\lambda_0$, such that for any $r\ge r_0$, $R\ge r$, and $0\le \lambda \le \lambda_0$, 
$$
 \sup_{R\le \|z\|\le 2R} \mathbb E_z \Big[\exp\big(\lambda\frac{|\eta_{r,2R}|}{r^2\log(R/r)}\big)\Big] \le \exp\big(\frac{c\lambda}{R^2\log(R/r)}\big). $$ 
\end{proposition}
Once this proposition is established, the end of the proof of Theorem~\ref{theo.outside4} is almost the same as in dimension five and higher. Indeed, 
let us postpone the proof of Proposition~\ref{prop.phi.d4} for a moment, and conclude the proof of Theorem~\ref{theo.outside4} first. 

\begin{proof}[Proof of Theorem~\ref{theo.outside4}] 
Let $r\ge r_0$, with $r_0$ as in Proposition~\ref{prop.phi.d4}, and $R\ge 2r$ be given, and let also $x$ be such 
that $2r\le \|x\|\le R$. Let $i_0$ be the smallest integer, 
such that $\|x\|\le r2^{i_0}$, 
and let $I$ be the smallest integer such that $R\le r2^{i_0+I}$. 
Recall the definition \reff{def-R} of $R_i$ and $Z_{r,i}$, and notice that 
$$
|\eta_{r,R}| \le \sum_{i=0}^I Z_{r,i}.
$$ 
Moreover, as in the proof of Theorem~\ref{theo.outside5}, on one hand Proposition~\ref{prop.phi.d4} shows that (with the same constant $c$ as in its statement), 
$$
\mathbb E_x\Big[\exp\big(\frac{\lambda}{r^2\log(R/r)}
\sum_{i=0}^I Z_{r,i}\big)\Big] \le\exp\big(\frac{c\lambda}{\|x\|^2\log(R/r)}\big)
\cdot \mathbb E_x\Big[ \exp\big(\frac{2c\lambda}{\log(R/r)}
\cdot \sum_{i=0}^{I-1} \frac{ |\eta_{R_i}|}{R_i^2}\big)\Big]^{1/2},
$$
and on the other hand, using~\eqref{inside3}, we get by induction
$$ \mathbb E_x\Big[ \exp\big(\frac{2c\lambda}{\log(R/r)}\cdot \sum_{i=0}^{I-1} 
\frac{ |\eta_{R_i}|}{R_i^2}\big)\Big] 
\le \exp\big(\frac{8c\lambda}{\|x\|^2 \log(R/r)}\big),$$
which concludes the proof of~\eqref{outside4.1}. The proof of~\eqref{outside4.2} follows exactly as for the corresponding statement 
in Theorem~\ref{theo.inside}, namely~\eqref{inside1}. Indeed, using a similar argument as in the proof of Lemma~\ref{lem.hitting}, 
one can see that uniformly over $R\ge 4r$, one has $\inf_{x\in \partial B_{2r}} \mathbb P(\eta_{r,R}\neq \varnothing) \ge c/r^2$, for some constant $c>0$. Finally concerning the proof of~\eqref{outside4.3}, 
one can argue as follows. First using the first point of Lemma~\ref{lem.mom.etar}, one has for some constant $\varepsilon>0$, 
\begin{align*}
\mathbb E_x\Big[\exp\big(\frac{\lambda|\eta_{r,R}|}{r^2\log(R/r)}\big)\Big]  \le &\ 1+\frac{(1-\varepsilon)\lambda}{r^2\log(R/r)} + \frac{\lambda^2}{r^4\log(R/r)^2} 
\mathbb E_x\Big[|\eta_{r,R}|^2\exp\big(\frac{2\lambda|\eta_{r,R}|}{r^2\log(R/r)}\big)\Big]. 
\end{align*}
Then, using Proposition~\ref{prop.hit.upper} and H\"older's inequality at the third line, and \eqref{outside4.2} and the second point of Lemma~\ref{lem.mom.etar} at the last one, we get  
\begin{align*}
& \mathbb E_x\Big[\exp\big(\frac{\lambda|\eta_{r,R}|}{r^2\log(R/r)}\big)\Big]\\
 & \le 1+\frac{(1-\varepsilon)\lambda}{r^2\log(R/r)} + \frac{\lambda^2}{r^4\log(R/r)^2} 
\mathbb E_x\Big[|\eta_{r,R}|^2\exp\big(\frac{\lambda|\eta_{r,R}|}{r^2\log(R/r)}\big)\ 
\Big|\ \eta_{r,R}\neq \emptyset\Big]\cdot \mathbb P_x(|\eta_{r,R}|>0) \\
& \le 1+\frac{(1-\varepsilon)\lambda}{r^2\log(R/r)} + \frac{c\lambda^2}{r^{4+2/3}\log(R/r)^2} 
\mathbb E_x\big[|\eta_r|^3\big]^{2/3} \cdot 
\mathbb E_x\Big[\exp\big(\frac{3\lambda|\eta_{r,R}|}{r^2\log(R/r)}\big)\ 
\Big|\ \eta_{r,R}\neq \emptyset\Big]^{1/3}\\
& \le 1+\frac{(1-\varepsilon)\lambda}{r^2\log(R/r)} + \frac{c\lambda^2}{r^2\log(R/r)^2}\le \exp\big(\frac{(1-\varepsilon')\lambda}{r^2\log (R/r)}\big), 
\end{align*} 
for some constant $\varepsilon'\in(0,1)$, proving well~\eqref{outside4.3}. This concludes the proof of Theorem~\ref{theo.outside4}. 
\end{proof}
It remains now to prove Proposition~\ref{prop.phi.d4}, which requires some more care than for the proof of Proposition~\ref{prop.phi.i}. 
In particular one needs first an improved version of Lemma \ref{lem.phir3}. 

\begin{lemma}\label{lem.phir4}
Assume $d=4$. There exists positive constants $c$ and $\lambda_0$, such that for any $0\le \lambda, \lambda'\le \lambda_0$, and any $r\ge 1$, 
$$\sup_{x\in \partial B_{2r}} \mathbb E_x \Big[\exp\big(\lambda\frac{|\eta_{r,4r}|}{r^2} + 
\lambda' \frac{|\eta_{4r,r}|}{4r^2}\big)\Big] \le \exp\Big(\frac{\lambda(1+\frac{c}{r}+c\lambda)+ \lambda'(1+\frac{c}{r}+c\lambda')}{5r^2}\Big). $$ 
\end{lemma}
\begin{proof}
The proof is similar to the proof of Lemma~\ref{lem.phir3}. First we write for any $x\in \partial B_{2r}$, 
$$\mathbb E_x\Big[\exp\big(\frac{\lambda |\eta_{r,4r}|}{r^2}\big)\Big] \le 1+ \frac{\lambda \mathbb E_x\big[|\eta_{r,4r}|\big]}{r^2} + 
\frac{\lambda^2}{r^4} \mathbb E_x\Big[|\eta_{r,4r}|^2\exp\big(\frac{\lambda |\eta_{r,4r}|}{r^2}\big)\Big],$$ 
and
$$\mathbb E_x\big[|\eta_{r,4r}|\big] = \mathbf P_x(H_r<H_{4r}) \le  \frac{G(x) - G(4r)}{G(r)- G(4r)},$$
with $G(s)= \inf_{z\in \partial B_s} G(z)$, for $s\ge 1$. Using \eqref{Green}, we deduce that in dimension four, for some constant $c>0$, 
$$\sup_{x\in \partial B_{2r}} \mathbb E_x\big[|\eta_{r,4r}|\big] \le \frac{1}{5} (1+ c/r).$$ 
We conclude as in the proof of Lemma~\ref{lem.phir3} that for $\lambda$ small enough,  
 $$\sup_{x\in \partial B_{2r}} \mathbb E_x\Big[\exp\big(\frac{\lambda |\eta_{r,4r}|}{r^2}\big)\Big] \le \exp\big(\frac{\lambda(1+\frac{c}{r}+c\lambda)}{5r^2}\big). $$ 
 The same argument leads to 
 $$\sup_{x\in \partial B_{2r}} \mathbb E_x\Big[\exp\big(\frac{\lambda' |\eta_{4r,r}|}{4r^2}\big)\Big] \le \exp\big(\frac{\lambda'(1+\frac{c}{r}+c\lambda')}{5r^2}\big),$$
 and the lemma follows using Cauchy-Schwarz inequality.  
\end{proof}

\begin{proof}[Proof of Proposition~\ref{prop.phi.d4}] 
Given $r>0$, let $R_i = r2^i$, for $i\ge 1$. We will prove the existence of positive constants 
$c$, $r_0$, and $\lambda_1>0$, such that for all $r\ge r_0$, $0\le \lambda\le \lambda_1$, and $i\ge 1$, 
\begin{equation}\label{goal.prop81}
\sup_{x\in  \partial B_{R_i}}\mathbb E_x\Big[\exp\big(\frac{\lambda |\eta_{r,R_{i+1}}|}{ir^2}\big)\Big] \le \exp\big(\frac{c\lambda}{iR_i^2}\big). 
\end{equation}
Note that exactly as in the proof of Proposition~\ref{prop.phi.i}, the desired result would follow. 

Now consider the functions $\psi_{r,i}$, defined for $\lambda,\lambda'\ge 0$ by  
$$\psi_{r,i}(\lambda,\lambda'):= \sup_{x\in  \partial B_{R_i}}\mathbb E_x\Big[\exp\big(\frac{\lambda |\eta_{r,R_{i+1}}|}{ir^2}+\frac{\lambda' |\eta_{R_{i+1},r}|}{iR_i^2}\big)\Big].$$ 
We claim that there exist positive constants $r_0$, $\kappa$ and $\lambda_1\le 1$, such that for all $r\ge r_0$, all $i\ge 1$, and all $0\le\lambda,  \lambda'\le \lambda_1$, 
\begin{equation}\label{induction.d4}
\psi_{r,i}(\lambda,\lambda') \le \exp\Big(   \frac{\alpha_i(\lambda+\lambda')\cdot \lambda+\beta_i(\lambda+\lambda')\cdot\lambda'}{iR_i^2}\Big),
\end{equation}
writing for $t\ge 0$, 
$$\alpha_i(t) := \prod_{j=1}^i(1+\frac{\kappa }{R_j}+\frac {\kappa t} i), \quad \text{and}\quad \beta_i(t)=1+\frac{\kappa  }{R_i}+\frac{\kappa t}{i}.$$ 
Note that once $\kappa$ is fixed, one has $\sup_i \alpha_i(1)<\infty$, hence taking $ \lambda'=0$ in~\eqref{induction.d4} gives~\eqref{goal.prop81}, and thus concludes the proof of the Proposition.

We prove~\eqref{induction.d4} by induction, and start by fixing the constants $\kappa$, $r_0$ and $\lambda_1$. For this define for $i\ge 0$, with $c$ the constant appearing in Lemma~\ref{lem.phir4}, 
$$\rho_i : = 1+\frac{c}{R_i} + \frac{10c(\lambda + \lambda')}{i+1},\quad \text{and}\quad \nu_i:= 1+\frac{\kappa }{R_i} + \frac{\kappa(\lambda + \lambda')}{i+1}.$$  
Then fix $\kappa$ and $r_0$ large enough, and $\lambda_1\le 1$ small enough, such that for all $r\ge r_0$, and all $0\le \lambda, \lambda'\le \lambda_1$, 
\begin{equation}\label{rhoinui1}
\rho_0 \le 5/3,   \qquad \nu_0\le 2, \qquad \sup_i \alpha_i(2\lambda_1)\le 2, \qquad \lambda_1\le \frac{\lambda_0}{10}, 
 \end{equation}
 where $\lambda_0$ is the constant given by Lemma~\ref{lem.phir4}, and
\begin{equation}\label{rhoinui2}
 \frac{5\rho_i}{5-\rho_i\nu_i} \le \frac{5}{4}\cdot \big(1+\frac{\kappa }{R_{i+1}} + \frac{\kappa(\lambda+\lambda')}{i+1}\big). 
 \end{equation}
To be more precise, one may first choose $\kappa=100c$, and then take $r_0$ large enough, and $\lambda_1$ small enough 
so that~\eqref{rhoinui1} and~\eqref{rhoinui2} are satisfied, for all $r\ge r_0$, $i\ge 0$ and $0\le \lambda,\lambda'\le \lambda_1$.
Now~\eqref{induction.d4} for $i=1$ is given by Lemma~\ref{lem.phir4}, at least provided 
$c\le \kappa/2$, which we can always assume.
Assume next that~\eqref{induction.d4} holds for some $i\ge 1$, and let us prove it for $i+1$. Define $\{\zeta_k^j\}_{k\ge 0}$, for $j\in \{0,1,2,3\}$, 
as well as $\{\mathcal G_k\}_{k\ge 0}$ and $\{\mathcal H_k\}_{k\ge 0}$, as in the proof of Proposition~\ref{prop.phi.i}. 
Recall that by monotone convergence, one has 
\begin{equation}\label{psi.conv.mon}
\psi_{r,i+1}(\lambda,\lambda') \le \lim_{n \to \infty} \sup_{x\in \partial B_{R_{i+1}}}
\mathbb E_x\Big[\exp\big(\frac{\lambda\sum_{k=0}^n |\zeta_k^0|}{(i+1)r^2}+\frac{\lambda'\sum_{k=0}^n |\zeta^3_k|}{(i+1)R_{i+1}^2}\big)\Big]. 
\end{equation} 
The induction hypothesis implies that almost surely, 
$$\mathbb E_x\Big[\exp\big(\frac{\lambda |\zeta_n^0|}{(i+1)r^2}\big) \ \big|\ \mathcal G_n \Big] \le 
\exp\big( \frac{\widetilde \alpha_i \lambda|\zeta_n^1|}{(i+1)R_i^2}\big),\qquad \text{with }\widetilde \alpha_i: = \alpha_i(\frac{i(\lambda+\lambda')}{i+1}). $$ 
Note in particular that $(\sup_i \widetilde \alpha_i)\lambda \le 2\lambda \le \lambda_0$, by~\eqref{rhoinui1}. Hence, an application of Lemma~\ref{lem.phir4} yields that almost surely 
$$\mathbb E_x\Big[\exp\big(\frac{\widetilde \alpha_i\lambda |\zeta_n^1|}{(i+1)R_i^2}+ \frac{\lambda' |\zeta^3_n|}{(i+1)R_{i+1}^2}\big)
 \ \big|\ \mathcal H_n \Big] \le 
\exp\big(\frac{\rho_i( \widetilde \alpha_i \lambda  + \lambda') |\zeta_n^2|}{5(i+1)R_i^2}\big). $$  
Observe next that by~\eqref{rhoinui1}, one has 
$$\frac{\rho_i}{5}(\widetilde \alpha_i \lambda  + \lambda')\le  \lambda_1.$$
Thus applying again the induction hypothesis, we get that almost surely, 
$$\mathbb E_x\Big[\exp\big(\frac{\lambda |\zeta_{n-1}^0|}{(i+1)r^2}+ \frac{\rho_i( \widetilde \alpha_i \lambda  + \lambda') |\zeta_n^2|}{5(i+1)R_i^2}\big) \ \big|\ \mathcal G_{n-1} \Big] 
\le \exp \Big(\big\{  \widetilde \alpha_i \lambda(1 +\frac{\rho_i\nu_i}{5}) + \frac{\rho_i\nu_i\lambda'}{5}\big\} \cdot \frac{|\zeta_{n-1}^1|}{(i+1)R_i^2}\Big). $$  
Now, one has by~\eqref{rhoinui1} again 
$$\widetilde \alpha_i(1+\frac{\rho_i\nu_i}{5})\lambda + \frac{\rho_i\nu_i}{5}\lambda'  \le \big(2(1+4/5) +4/5\big)\lambda_1 \le 10\cdot \lambda_1 \le \lambda_0.$$
Hence one may apply Lemma~\ref{lem.phir4}, which gives (note the factor $10$ appearing in the definition of $\rho_i$), 
\begin{align*}
& \mathbb E_x\Big[\exp \Big(\big\{  \widetilde \alpha_i\lambda(1 +\frac{\rho_i\nu_i}{5}) + \frac{\rho_i\nu_i\lambda'}{5}\big\}  \cdot \frac{|\zeta_{n-1}^1|}{(i+1)R_i^2}+ \frac{\lambda'|\zeta_{n-1}^3|}{(i+1)R_{i+1}^2}\Big)\ \Big| \ \mathcal G_{n-1} \Big] \\
& \le \exp\Big( \big( \widetilde \alpha_i\lambda  + \lambda'\big)\cdot(1+\frac{\rho_i\nu_i}{5})\cdot \frac{ \rho_i |\zeta_n^2|}{5(i+1)R_i^2}\Big). 
\end{align*}
Then an elementary induction shows that for all $n\ge 1$, 
$$\mathbb E_x\Big[\exp\big(\frac{\lambda\sum_{k=0}^n |\zeta_k^0|}{(i+1)r^2}+\frac{\lambda'\sum_{k=0}^n |\zeta^3_k|}{(i+1)R_{i+1}^2}\big)\Big] 
\le \exp\Big(\frac{4\rho_i}{5}\cdot \frac{ \widetilde \alpha_i\lambda+\lambda'}{(i+1)R_{i+1}^2}\cdot \sum_{k=0}^n (\frac{\rho_i\nu_i}{5})^{k}\Big).$$
To conclude, note that by~\eqref{rhoinui2}, 
$$\rho_i\sum_{k=0}^\infty (\frac{\rho_i\nu_i}{5})^{k} = \frac{5\rho_i}{5-\rho_i\nu_i} \le \frac{5}{4}\cdot \big(1+\frac{\kappa }{R_{i+1}} + \frac{\kappa(\lambda+\lambda')}{i+1}\big) =\frac 54 \beta_{i+1}(\lambda+\lambda').$$
Finally notice that $\widetilde \alpha_i \beta_{i+1}(\lambda+\lambda') = \alpha_{i+1}(\lambda+\lambda')$, thereby finishing the proof of the induction 
step for \eqref{induction.d4}. 
This concludes the proof of Proposition~\ref{prop.phi.d4}. 
\end{proof}


\section{Proof of Theorem~\ref{theo.main}: upper bounds in $d\ge 4$}
\label{sec-UB4}

The proof in dimension four is slightly different from 
the higher dimensional case, but first one needs 
to strengthen the result of Proposition~\ref{prop.Tr} as follows. 
Recall that we denote by $\mathcal T(B_r)$ the time spent in the ball $B_r$ 
by a BRW, 
for which we freeze the particles reaching the 
boundary of the ball, see~\eqref{def.Tr}. 
\begin{proposition}\label{prop.Tr+}
Assume $d\ge 1$. 
There exist positive constants $c$ and $C$, such that for all $r\ge 1$, and all $t\ge r^4$, 
$$\sup_{x\in B_r} \mathbb P_x (|\mathcal T(B_r)|>t) \le \frac{C}{r^2} \exp(-ct/r^4). $$
\end{proposition}
\begin{proof}
Given $u\in \mathcal T$, and $\Lambda\subset \Z^d$, we 
call $\mathcal T^u(\Lambda)$ the random variable with the same law 
as $\mathcal T(\Lambda)$ shifted in the subtree emanating from $u$. 
Recall also that $\mathcal T_n=\{u\in  \mathcal T: |u|\le n\}$, and assume without loss of generality that $r^2$ is an integer. 
One has 
\begin{equation}\label{Tr.rec}
\mathbb P_x(|\mathcal T(B_r)|>t)  \le \mathbb P_x\Big(\sum_{|u|=r^2} 
|\mathcal T^u(B_r)|\ge t/2\Big)+  
\mathbb P\big(|\mathcal T_{r^2}|>t/2\big).
\end{equation}
For the first term on the right-hand side we use first Proposition~\ref{prop.Tr} and Chebyshev's exponential inequality, 
which give for any $\lambda$ small enough, 
$$
\mathbb P_x\Big(\sum_{|u|=r^2} |\mathcal T^u(B_r)|\ge t/2\Big) \le e^{- \frac{\lambda t}{2r^4}} \cdot \mathbb E\Big[\exp\big(c\lambda \frac{Z_{r^2}}{r^2}\big)\ \Big|\ Z_{r^2}\neq 0\Big]\cdot \mathbb P(Z_{r^2}\neq 0) . $$ 
Then Lemma~\ref{lem.BGW} and \eqref{Kol} yield for some constant $C>0$, 
$$
\mathbb P_x\Big(\sum_{|u|=r^2} |\mathcal T^u(B_r)|\ge t/2\Big) 
\le \frac{C}{r^2} \exp(-\frac{\lambda t}{2r^4}).
$$
Concerning now the second term on the right-hand side of~\eqref{Tr.rec}, we use a similar idea. 
Let $I$ be the smallest integer, such that $r^2\le 2^I$. Using this time both Lemmas~~\ref{lem.BGW} and~\ref{lem.BGW2} shows that 
for some $\lambda$ small enough, and positive constants $c$, and $C$, 
\begin{align*}
\mathbb P\big(|\mathcal T_{r^2}|>t/2\big) & \le \sum_{i=0}^{I-1} \mathbb P\big(\sum_{k=2^i}^{2^{i+1}} Z_k >2^{i-I-1} t\big)\\
& \le \sum_{i=0}^{I-1} \exp\big(-\frac{\lambda 2^{i-I-1}t}{2^{2i}}\big)\cdot \mathbb E\Big[\exp\big(\frac{c\lambda Z_{2^i}}{2^i}\big)\ \Big|\ Z_{2^i}\neq 0\Big]\cdot \mathbb P(Z_{2^i}\neq 0) \\
& \le \frac{C}{r^2} \sum_{i=0}^{I-1}2^{I-i} \exp\big(-\frac{\lambda 2^{I-i}t}{8r^4}\big) \le \frac{C}{r^2} \exp(-\frac{\lambda t}{8r^4}), 
\end{align*}
using $t\ge r^4$, for the last inequality. This concludes the proof of the proposition. 
\end{proof}

We can now finish the proof of the upper bounds in Theorem~\ref{theo.main} 
in dimensions four and higher. 

Assume first that $d\ge 5$. Note that it suffices to prove the result for $r$ large enough. In particular we assume now that $r\ge r_0$, with $r_0$ given by Theorem~\ref{theo.outside5}. Define two sequences $\{\zeta_k^j\}_{k\ge 0}$, for $j=1,2$, by $\zeta_0^1 = \{\emptyset\}$ is the root of the tree, and for $k\ge 0$, 
$$
\zeta_k^2 := \bigcup_{u\in \zeta_k^1} \eta^u_{2r}, 
\quad \text{and}\quad \zeta_{k+1}^1 := \bigcup_{u\in \zeta_k^2} \eta^u_r.
$$ 
Let also for $k\ge 0$, 
$$L_k(B_r) := \sum_{u\in \zeta_k^1} |\mathcal T^u(B_{2r})|, $$ 
with the notation from the proof of the previous proposition. Then by definition, one has 
$$\ell_{\mathcal T}(B_r) \le \sum_{k\ge 0} L_k(B_r).$$
Therefore by Proposition~\ref{prop.hit.inside}, for any $\lambda>0$, 
$$
\mathbb P\big(\ell_{\mathcal T}(B_r)>t\big) \le 
\mathbb P\big(|\mathcal T(B_{2r})|>t/2\big) + 
\frac{C}{r^2}\exp(-\frac{\lambda t}{2r^4})\cdot \mathbb E\Big[\exp\big(\frac{\lambda}{r^4}\sum_{k\ge 1}  L_k(B_r)\big)\ \Big|\ |\eta_r|> 0 \Big]. $$ 
The first term on the right-hand side is handled by Proposition~\ref{prop.Tr+}. Hence only the last expectation is at stake, and it just 
amounts to show that it is bounded. By monotone convergence one has 
$$\mathbb E\Big[\exp\big(\frac{\lambda}{r^4}\sum_{k\ge 1}  L_k(B_r)\big)\ \Big|\ 
|\eta_r|> 0 \Big] = \lim_{n\to \infty} \mathbb E\Big[\exp\big(\frac{\lambda}{r^4}\sum_{k=1}^n  L_k(B_r)\big)\ \Big|\ |\eta_r|> 0 \Big]. $$ 
Note also that if $\mathcal G_n$ denotes the 
sigma-field generated by the BGW tree cut at vertices in $\zeta_n^1$, 
together with the positions of the BRW on this subtree, then by Proposition~\ref{prop.Tr}, almost surely for all $\lambda$ small enough, 
$$
\mathbb E\Big[\exp\big(\frac{\lambda L_n(B_r)}{r^4}\big)\ \Big|\ \mathcal G_n  \Big] \le \exp\big(\frac{c\lambda |\zeta_n^1|}{r^2}\big),$$
for some constant $c>0$. It follows that for any $n\ge 0$,  
$$
\mathbb E\Big[\exp\Big(\frac{\lambda}{r^4}\sum_{k=1}^n  L_k(B_r)  - \frac{c\lambda}{r^2}\sum_{k=1}^n |\zeta_k^1| \Big)\ \Big|\ |\eta_r|> 0 \Big]\le 1. $$ 
Hence by Cauchy-Schwarz, it just amounts to show that for $\lambda$ small enough, the sequence $\{u_n(\lambda)\}_{n\ge 1}$ defined by 
$$u_n(\lambda):=\mathbb E\Big[\exp\Big(\frac{c\lambda}{r^2}\sum_{k=1}^n |\zeta_k^1| \Big)\ \Big|\ |\eta_r|> 0 \Big], $$
is bounded. By combining~\eqref{inside3} and~\eqref{outside3}, we get that 
for $\lambda$ small enough, and some $\varepsilon>0$, one has for any $k\ge 2$, almost surely 
$$\mathbb E\Big[\exp\Big(\frac{\lambda}{r^2}|\zeta_k^1| \Big)\ \Big|\ \mathcal G_{k-1} \Big] \le \exp\big(\frac{(1-\varepsilon)\lambda}{r^2} |\zeta_{k-1}^1|\big),$$
and using also \eqref{inside1}, we get for some constant $c'>0$, 
$$ \mathbb E\Big[\exp\Big(\frac{\lambda}{r^2}|\zeta_1^1| \Big)\ \Big|\ 
|\eta_r|> 0 \Big] \le \exp(c'\lambda).$$
We deduce that for any $n\ge 1$, and $\lambda$ small,
that $\{u_n(\lambda)\}_{n\ge 1}$ is bounded through 
$$u_n(\lambda) \le \exp\big(c'\lambda 
\sum_{k=0}^n (1-\varepsilon)^k)\le \exp(c'\lambda/\varepsilon),$$
This concludes the proof of the upper bound in Theorem~\ref{theo.main}, 
in case $d\ge 5$.

In the case when $d=4$, the proof follows a similar pattern, but we
need to consider a truncated deposition process on $B_r$.
Instead of $\{\zeta_k^j\}_{k\ge 0}$, 
set $\tilde \zeta_0^1 = \{\emptyset\}$, and 
with $R = r2^I$, and $I:=\sqrt{t/r^4}$, we define for $k\ge 0$, 
$$
\tilde \zeta_k^2:= \bigcup_{u\in \tilde \zeta_k^1} \eta^u_{2r}, 
\quad \text{and}\quad 
\tilde \zeta_{k+1}^1: = \bigcup_{u\in \tilde \zeta_k^2} 
\eta^u_{r,R}.
$$ 
Then we define similarly for $k\ge 0$, 
$$
L_k(B_r) := \sum_{u\in \tilde \zeta_k^1} |\mathcal T^u(B_{2r})|,
$$
and Chebychev's inequality reads
\begin{equation}\label{d4-dist1}  
\begin{split}
\mathbb P\big(\ell_{\mathcal T}(B_r)>t\big) \le &
\mathbb P\big(|\mathcal T(B_{2r})|>\frac t2\big) +
 \mathbb P(|\eta_R|>0) \\
&\qquad + \frac{C}{r^2}e^{-\frac{\lambda t}{2Ir^4}} \cdot 
\mathbb E\Big[\exp\big(\frac{\lambda}{Ir^4}\sum_{k\ge 1}  L_k(B_r)\big)\ 
\Big|\ |\eta_r|> 0 \Big].  
\end{split}
\end{equation}
The second term is $\mathcal O(1/R^2)$ by Proposition~\ref{prop.hit.inside}, and the other terms are handled exactly as in higher dimension, using the estimates from Theorem~\ref{theo.outside4}, instead as from Theorem~\ref{theo.outside5}. \hfill $\square$


\section{Proof of Theorem~\ref{theo.main}: lower bounds in $d\ge 4$}
\label{sec-LB4}

We start with the case of dimension five and higher, 
and consider the subtle case of dimension four separately. 

\subsection{Dimension five and higher}
Assume $d\ge 5$, and recall that we may also assume here that $t\ge r^4$. 
The strategy we will use is to ask the BRW to reach $\partial B_r$ (at a cost of order $1/r^2$), 
and then make order $t/r^4$ excursions (or waves) between $\partial B_r$ and $\partial B_{r/2}$, at a cost of order $\exp(-\Theta(t/r^4))$.

More precisely, the proof relies on the following lemma, which holds in fact in any dimension (recall~\eqref{etar}). 
\begin{lemma}\label{lem.inf.etar} Assume $d\ge 1$. 
There exists a constant $c>0$, such that for any $r\ge 1$,
$$
\inf_{x\in B_{r/2}\cup \partial B_{2r}} \, \mathbb P_x( |\eta_r|\ge cr^2)  \geq c/r^2.
$$ 
\end{lemma}
\begin{proof} Let
$$\overline{\eta}_r := \{u\in \eta_r  :   r^2\le  |u|\le 2r^2\}. $$
Note that uniformly over $x\in B_{r/2} \cup \partial B_{2r}$, 
$$\mathbb E_x[|\overline{\eta}_r|] = \sum_{n=r^2}^{2r^2} \mathbb E[Z_n] 
\cdot \mathbf P_x(H_r = n ) = \mathbf P_x( r^2\le H_r\le 2r^2 ) \gtrsim 1, $$
and using a similar computation as in Subsection~\ref{sec.hitting}, 
\begin{equation*}
\mathbb E_x\big[|\overline{\eta}_r|^2\big] = 
\mathbb E_x[|\overline{\eta}_r|] + \sigma^2\,  \sum_{k=0}^{r^2-1} 
\mathbf E_x\big[\1\{H_r>k\}\cdot 
\mathbf P_{S_k}\big(r^2-k\le H_r\le 2r^2-k\big)^2\big] \le 1+\sigma^2 r^2. 
\end{equation*}
Now by~\eqref{Kol}, 
$$\mathbb P_x(|\overline{\eta}_r|>0 )  \le \mathbb P(Z_{r^2}\neq 0) \lesssim 1/r^2. $$
Therefore, 
$$\mathbb E_x\big[|\overline{\eta}_r| \mid |\overline{\eta}_r|>0\big] \gtrsim r^2, $$ 
and Paley-Zygmund's inequality~\eqref{PZ} gives, for some constant $c>0$, 
$$\mathbb P_x\left(|\overline{\eta}_r| > cr^2\right) \gtrsim 1/r^2.$$ 
Since $\overline{\eta}_r\subset \eta_r$, this proves the lemma. 
\end{proof}

Consequently, if we start with order $r^2$ particles on $\partial B_r$, then with probability of order $1$, there will be order $r^2$ particles reaching 
$\partial B_{r/2}$. Repeating this argument we see that with probability at least 
$\Theta(1/r^2) \cdot \exp(-\Theta(t/r^4))$, the BRW will
make at least $\Theta(t/r^4)$ waves between $\partial B_r$ and $\partial B_{r/2}$, with an implicit constant as large as wanted.  
The expected time spent on $B_r$ by each of these waves is of 
order $r^4$, simply because for any starting point on $\partial B_{r/2}$, the expected time spent on $B_r$ by the BRW killed on $\partial B_r$ is 
of order $r^2$. 
Since all the waves are independent of each other (at least conditionally on the positions of the frozen particles), we shall deduce  
that the total time spent on $B_r$ will exceed $t$.

On a formal level now, we define $\{\zeta_k^1\}_{k\ge 0}$, 
$\{\zeta_k^{1,1}\}_{k\ge 0}$ and $\{\zeta_k^2\}_{k\ge 0}$ inductively by 
$\zeta_0^{1,1} = \zeta_0^1 = \{\emptyset\}$ (the root of the tree), and then for $k\ge 0$, 
$$\zeta_k^2 = \bigcup_{u\in  \zeta_k^{1,1}} 
\eta^u_r, \qquad \zeta_{k+1}^1 = \bigcup_{u\in \zeta_k^2} 
\eta^u_{r/2}, $$
and for $\zeta_{k+1}^{1,1}$ we take any subset (chosen arbitrarily, for instance uniformly at random) of  $\zeta_{k+1}^1$ 
with $\lceil |\zeta_{k+1}^1|/2\rceil$ points. We also let $\zeta_k^{1,2}: = \zeta_k^1\setminus \zeta_k^{1,1}$, for $k\ge 0$. 
Next, set $N:=\lceil Ct/r^4\rceil$, with $C>0$ some large constant to be fixed later, and define further for $k\ge 1$, 
$$L_k(B_r):= \sum_{u\in \zeta_k^{1,2}} |\mathcal T^u(B_r)|,$$
with the notation from the proof of Proposition~\ref{prop.Tr+}. The reason why we partition the sets $\zeta_k^1$ into two parts, is that we want to keep, 
for each $k\ge 1$, some independence between $\zeta_k^2$ and $L_k(B_r)$, conditionally on $\zeta_k^1$. 
Now, note that 
\begin{equation}\label{ellinfty}
\ell_{\mathcal T}(B_r) \ge \sum_{k=1}^N L_k(B_r). 
\end{equation} 
For $k\ge 1$, define $\mathcal G_k$ as the sigma-field generated by tree cut at vertices in $\zeta_k^1$, together with the choice of $\zeta_k^{1,1}$ and the positions of the BRW on the vertices of this subtree. Let also $E_k:=\{|\zeta_k^{1,1}|\ge cr^2\}$, with $c$ chosen such that 
\begin{equation}\label{rho.def}
\rho:=\min\Big(\mathbb P(E_1), \inf_{k\ge 2} \mathbb P(E_k\mid E_{k-1})\Big)>0.
\end{equation} 
Note that the existence of $c$ is 
guaranteed by Lemma~\ref{lem.inf.etar}. Observe also that $E_k\in \mathcal G_k$, for any $k\ge 1$. 
Remember next that for some constant $c'>0$, one has for all $r\ge 1$, 
$$
\inf_{x\in \partial B_{r/2}} \mathbb E_x\big[|\mathcal T(B_r)|\big] = \inf_{x\in \partial B_{r/2}} \mathbf E_x[H_r]\ge c'r^2.$$ 
As a consequence, for any $k\ge 1$, almost surely, 
$$\mathbb E \big[L_k(B_r)\mid \mathcal G_k]\ge c'r^2|\zeta_k^{1,2}|. $$ 
On the other hand, 
$$\sup_{x\in B_r} \mathbb E_x\big[|\mathcal T(B_r)|\big] =\sup_{x\in B_r} \mathbf E_x[H_r] = \mathcal O(r^2),$$ 
which entails that for some constant $K>0$, for any $x\in B_r$, 
\begin{equation*}
\mathbb E_x\big[|\mathcal T(B_r)|^2\big]  
\le  \mathbb E_x \big[|\mathcal T(B_r)|\big] 
+ \sum_{k=0}^\infty \mathbb E_x\Big[\sum_{\substack{u\in  \mathcal T(B_r) \\ |u|=k}} \xi_u(\xi_u-1)\Big(\sup_{z\in B_r} 
\mathbb E_z\big[|\mathcal T(B_r)|\big]\Big)^2\Big]  
\le Kr^6. 
\end{equation*}
As a consequence, one has for any $k\ge 0$, 
$$\mathbb E\big[L_k(B_r)^2\mid \mathcal G_k\big] \le Kr^6 |\zeta_k^{1,2}|. $$ 
Using then Paley-Zygmund's inequality~\eqref{PZ1} we get 
\begin{equation}\label{PFk}
\mathbb P(F_k \mid \mathcal G_k)\1_{E_k}  \ge \frac{cc'^2}{4K}\1_{E_k}, \quad \text{with}\quad F_k:= \Big\{L_k(B_r)\ge \frac{cc'r^4}{2}\Big\}. 
\end{equation} 
Note now that on the event 
$$\tilde E_N := \bigcap_{k=1}^N E_k \cap F_k,$$
by~\eqref{ellinfty} one has $\ell_{\mathcal T}(B_r) \ge N cc'r^4/2$, which is larger than $t$, provided the constant $C$ in the definition of $N$ is chosen 
large enough. Moreover, for each $k\ge 0$, conditionally on $\mathcal G_k$, $E_{k+1}$ and $F_k$ are independent. Therefore by~\eqref{rho.def} and~\eqref{PFk}, one has also by induction 
$$\mathbb P(\tilde E_N) \ge \kappa^N,$$
with $\kappa=\rho cc'^2/4K$, concluding the proof of the lower bound in dimension five and higher. \hfill $\square$ 

\subsection{Dimension Four}
We assume in this section that $d=4$ and $t\ge r^4$. 
Since our scenario producing a lower bound is new, we first present
heuristics, followed by the formal proof. 

\paragraph{Heuristics.}
Our strategy for producing a local time $t$ in $B_r$ 
is very much different than the scenario presented in $d\ge 5$. 
Recall
two facts specific to dimension four, when the BRW starts from $\partial B_R$, with $R>r$: (i) first Theorem~\ref{theo.outside4} shows that $\eta_{r,2R}$ 
is typically of order
$r^2\cdot \log(R/r)$, which is much larger than the corresponding
number in $d\ge 5$; (ii) the probability of $\{\mathcal T(B_r) \neq 
\emptyset\}$ is smaller (by a factor $\log$) than the corresponding
probability in $d\ge 5$, and is of order $R^{-2}\cdot\log^{-1}(R/r)$.
With these two facts in mind, let us start with the heuristics and
set up notation.
Set $I = 2\lfloor C\sqrt{t/r^4}\rfloor $, with $C>0$ some large constant to be fixed later.   
Then for $i\le I$, let $R_i: = r 2^i$, and    
$$\mathcal S_i:= \{z\in \mathbb Z^d : R_{i-1} \le \|z\|\le R_i\}.
$$ 
Our first requirement will be for the BRW to reach distance $R_I$, 
and even more that $|\eta_{R_I}|$ be of order $R_I^2$. 
Lemma~\ref{lem.inf.etar} ensures that this has probability $\Theta(1/R_I^2)$, which is the expected cost. 
Next, by Lemma~\ref{lem.hitting} (or the second point (ii) recalled above), with probability $\Theta(1/I)$, one of the BRWs emanating from one of the vertices in $\eta_{R_I}$ will reach $\partial B_{r/2}$.
Furthermore, conditionally on this event, 
we know by Proposition~\ref{prop.Zhu.spine} 
that one spine reaches $\partial B_{r/2}$, and brings there of the order
of $r^2\cdot I$ walks, in virtue of the point (i) recalled above. 
Since from any of the vertices of the spine start independent BRWs, 
and since, as we will show, the spine typically spends a time 
of order $R_I^2$ in the shell $\mathcal S_I$, 
we deduce that at least one of the BRWs starting from the 
spine in this shell will reach as well $\partial B_{r/2}$, 
with probability $\Theta(1/I)$ again. 
Conditionally on this, one has now two spines 
crossing the shell $\mathcal S_{I-1}$. 
The probability that one of the BRWs starting from one of these two spines in $\mathcal S_{I-1}$ reaches $\partial B_{r/2}$ is thus of 
order twice $\Theta(\frac 1{I-1})$. Then by repeating this argument in all the shells $\mathcal S_I,\dots,\mathcal S_{I/2}$, 
one can make sure that $I/2$ spines reach $\partial B_{r/2}$, at a total cost of only 
$(1/R_I^2)\cdot \exp(-\Theta(I))$, which is still affordable. 
To conclude we know that the spines typically come with 
order $r^2(I/2+\dots+I)\ge r^2I^2/4$ walks on $\partial B_{r/2}$.
Since each of them leads afterwards to a mean local time order 
$r^2$ in $B_r$, this concludes the heuristics.  We show in 
Figure~\ref{fig:dessin-d4} the many spines originating from successive
shells: the green spine gives birth to orange critical trees producing
an orange spine which in turn gives birth to purple trees producing
a purple spine, and so on and so forth.

\begin{figure}[htpb]
\centering
\includegraphics[width=8cm,height=10cm]{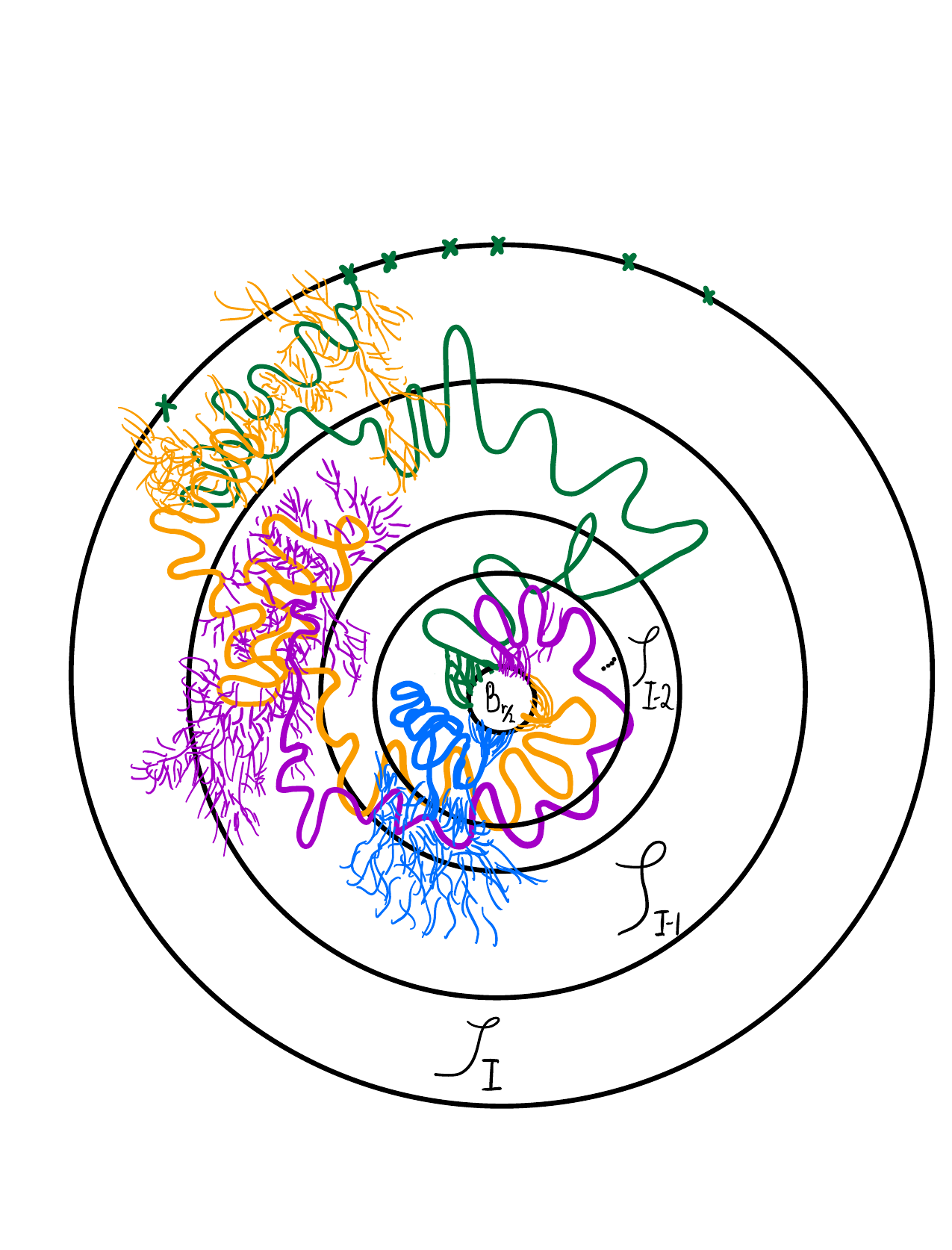}
\caption{The Many Spines.}\label{fig:dessin-d4}
\end{figure}

\paragraph{Proof.}
The formal proof follows very much the picture we just presented. 
Define the event $E_0:= \{|\eta_{R_I}|\ge cR_I^2\},$
with $c$ as in Lemma~\ref{lem.inf.etar}. Let
$$
E_1:= \big\{\sum_{u \in \eta_{R_I}}  
 |\eta^u_{r/2}|>0\big\}.
$$
Conditionally on $E_1$, we know by Proposition~\ref{prop.Zhu.spine} that there exists a path (or spine), 
which we denote by $\Gamma_1=(\Gamma_1(i),\ 0\le i \le |\Gamma_1|)$, 
emanating from one of the points in $\{S_u, u\in \eta_{R_I}\}$, 
and going up to $\partial B_{r/2}$. 
For $1\le j\le I$, and a path $\gamma$, we let  
$$
\tau_j = \tau_j(\gamma):= \inf\big\{ k\ge 0 : \gamma(k) \in \partial \mathcal S_j\big\}. $$
We then define $E_2$ as the event that one of the biased BRWs starting from the points in the path $\{\Gamma_1(i),\ 0\le i < \tau_{I-1}\}$, 
hits $\partial B_{r/2}$. Applying Proposition~\ref{prop.Zhu.spine} again, it means that 
on the event $E_1\cap E_2$ there exists a second spine $\Gamma_2$ emanating from one of the (neighbors of the) points in 
$\{\Gamma_1(i),\ 0\le i < \tau_{I-1}\}$, going up to $\partial B_{r/2}$. 
We can thus define $E_3$ as the event that one of the biased BRWs 
starting from the points in 
$\{\Gamma_1(i),\ \tau_{I-1}\le i < \tau_{I-2}\} 
\cup \{\Gamma_2(i),\ 0\le i < \tau_{I-2}\}$, 
hits $\partial B_{r/2}$. Continuing like this in all the 
shells $\mathcal S_I, \dots, \mathcal S_{I/2}$, we define inductively 
for each $i=1,\dots,I/2$, an event $E_i$, such that on $E_1\cap \dots\cap E_i$, there are $i$ spines $\Gamma_1,\dots,\Gamma_i$, starting respectively 
from a point in $\partial B_{R_I}, \partial \mathcal S_I, \dots, \partial \mathcal S_{I-i+2}$, and going up to $\partial B_{r/2}$. 
We claim that there exists a constant $\rho>0$, such that for each $i\le I/2$, 
\begin{equation}\label{hit.spine.claim}
\mathbb P(E_{i+1} \mid E_1\cap \dots \cap E_i) \ge \frac{\rho i}{I}. 
\end{equation}
Indeed, recall that by Corollary~\ref{cor.spine.Markov} the spines satisfy the strong Markov property. Therefore, for any $i$, 
$$\mathbb P(E_{i+1} \mid E_1\cap \dots \cap E_i) \ge 1 - \big(1-\inf_{x\in \partial B_{R_{I-i+1}}} r_i(x)\big)^i,$$
where for any $x\in \partial B_{R_{I-i+1}}$, we denote by $r_i(x)$ the probability that on a path $\Gamma$ starting from $x$, 
and sampled according to the measure $\mathbb P^x_{\partial B_{r/2}}$, one of the biased BRW starting from the points in $\{\Gamma(k),\ k\le \tau_{I-i}\}$, hits 
$\partial B_{r/2}$. 
Since for any $z\in \mathcal S_{I-i+1}$ the probability 
to hit $\partial B_{r/2}$ is of order $I^{-1}\cdot R_{I-i}^{-2}$ 
(recall that we assume $i\le I/2$), by Lemma~\ref{lem.hitting} (which holds as well for a biased BRW, see Remark~\ref{rem.lb.biased}), one has for some constant $c>0$, and any $\alpha>0$
$$r_i(x) \ge \Big\{1 - \big(1-\frac{c}{IR_{I-i}^2}\big)^{\alpha R_{I-i}^2}\Big\} \cdot \mathbb P_{\partial B_{r/2}}^x\big(\ell_\Gamma(\S_{I-i+1}) 
> \alpha R_{I-i}^2\big), 
$$
where for $j\ge 2$, we write $\ell_\gamma(\S_j)$ 
for the time spent on $\mathcal S_j$ by a path $\gamma$, before its hitting time of $B_{R_{j-1}}$. 
Thus, to conclude the proof of~\eqref{hit.spine.claim}, it suffices to show that for some $\alpha>0$, one has for any $j\ge 2$, and any 
$x\in \partial B_{R_j}$, 
\begin{equation}\label{elljGamma}
\mathbb P_{\partial B_{r/2}}^x\big(\ell_\Gamma(\S_j) > \alpha R_j^2\big) \ge \alpha. 
\end{equation}
Set now $\Lambda:=\partial B_{r/2}$ to simplify notation. One has by definition of $\mathbb P_\Lambda^x$, and using also \eqref{def.pLambda}, 
\begin{align*}
\mathbb P_\Lambda^x\big(\ell_\Gamma(\S_j)\le \alpha R_j^2\big) & =  \frac{\sum_{y\in \partial B_{R_{j-1}}} \left(\sum_{\gamma_1:x \to y}
\1\{\ell_\gamma(\S_j)\le \alpha R_j^2\} 
p_\Lambda(\gamma_1)\right)\cdot \left(\sum_{\gamma_2:y\to \Lambda} p_\Lambda(\gamma_2)\right)}
{\sum_{y\in \partial B_{R_{j-1}}} \left(\sum_{\gamma_1:x \to y} p_\Lambda(\gamma_1)\right)\cdot \left(\sum_{\gamma_2:y\to \Lambda} p_\Lambda(\gamma_2)\right)}\\
& = \frac{\sum_{y\in \partial B_{R_{j-1}}} \left(\sum_{\gamma_1:x \to y}
\1\{\ell_{\gamma_1}(\S_j)\le \alpha R_j^2\} p_\Lambda(\gamma_1)\right)\cdot 
\mathbb P_y(|\eta_{r/2}|>0)}
{\sum_{y\in \partial B_{R_{j-1}}} 
\left(\sum_{\gamma_1:x \to y} p_\Lambda(\gamma_1)\right)\cdot \mathbb P_y(|\eta_{r/2}|>0)}\\
& \le C \frac{\sum_{y\in \partial B_{R_{j-1}}} \sum_{\gamma_1:x \to y}
\1\{\ell_{\gamma_1}(\S_j)\le \alpha R_j^2\} p_\Lambda(\gamma_1)}
{\sum_{y\in \partial B_{R_{j-1}}}
\sum_{\gamma_1:x \to y} p_\Lambda(\gamma_1)}, 
\end{align*} 
for some constant $C>0$, using Lemma~\ref{lem.hitting} and Proposition~\ref{prop.hit.upper} for the last inequality. By \eqref{def.pLambda}, 
one has 
\begin{align*}
\sum_{y\in \partial B_{R_{j-1}}}\ \sum_{\substack{\gamma_1:x \to y \\ 
\ell_{\gamma_1}(\S_j)\le \alpha R_j^2}} p_\Lambda(\gamma_1)
 \le \sum_{y\in \partial B_{R_{j-1}}}  
\mathbf P_x\big(\ell_S(\S_j) \le \alpha R_j^2,\, S_{\tau_{j-1}}=y\big)=  \mathbf P_x\big(\ell_S(\S_j) \le \alpha R_j^2,\, \tau_{j-1}<\infty \big), 
\end{align*}
while using also Remark~\ref{rem.hit.adjoint}, we get that for some constant $c>0$, 
\begin{align*}
\sum_{y\in \partial B_{R_{j-1}}} \sum_{\gamma_1:x \to y} p_\Lambda(\gamma_1) 
\ge \sum_{y\in \partial B_{R_{j-1}}} \ \sum_{\substack{\gamma_1:x \to y \\ |\gamma_1|\le R_j^2}} p_\Lambda(\gamma_1) \ge c \mathbf P_x\big(\tau_{j-1}(S)\le R_j^2) \ge c^2. 
\end{align*}
Now for any $\varepsilon>0$, one can find $\alpha$ small enough, such that 
$$\mathbf P_x\big(\ell_S(\S_j) \le \alpha R_j^2,\, \tau_{j-1}<\infty \big)\le \mathbf P_x\big(\ell_S(\S_j) \le \alpha R_j^2\big)\le  \varepsilon.$$
Altogether this proves \eqref{elljGamma}, and thus also~\eqref{hit.spine.claim}. Using also Lemma~\ref{lem.inf.etar}, it follows that for some positive constants $c$ and $\kappa$, 
\begin{equation}\label{prob.Ei}
\mathbb P(E_0\cap \dots \cap E_{I/2}) \ge \exp(-\kappa I)\times\mathbb P(E_0) \ge \frac{c}{R_I^2} \exp(-\kappa I). 
\end{equation}
We observe finally that on the event in the probability above, 
the number of particles which hit $\partial B_{r/2}$ dominates the sum of $I/2$ independent random variables $X_1,\dots,X_{I/2}$ distributed as $|\eta_{r/2}|$, 
under the conditional law $\mathbb P_z(\cdot \mid |\eta_{r/2}|>0)$, 
for some $z\in \partial B_{R_{I/2}}$. 
However, for any such starting point $z$, one has using the computation made in the proof of Lemma~\ref{lem.hitting}, together 
with Proposition~\ref{prop.hit.upper},  
$$
\mathbb E_z\big[|\eta_{r/2}|\mid |\eta_{r/2}|> 0\big] \gtrsim I r^2, \quad\text{and}\quad 
\mathbb E_z\big[|\eta_{r/2}|^2\mid |\eta_{r/2}|>0\big] \lesssim r^4I^2.$$

Thus Paley--Zygmund's inequality \eqref{PZ1} (applied to the law of $X_1+\dots + X_{I/2}$) gives the existence of $\alpha>0$, such that
$$\mathbb P\big(X_1+\dots + X_{I/2} \ge \alpha I^2r^2\big) \ge \alpha. $$ 
In other words, we just have proved that 
$$\mathbb P\big(\sum_{u\in \eta_{R_I}} 
|\eta^u_{r/2}|\ge \alpha r^2 I^2\big)\ge \frac{c}{r^2}\exp(-\kappa' I), $$ 
for some positive constants $c$ and $\kappa'$. The proof is now almost finished. To conclude, note that 
$$\inf_{x\in B_{r/2}}\mathbb E_x\big[|\mathcal T(B_r)|\big] = 
\inf_{x\in B_{r/2}} \mathbf E_x[H_r] \gtrsim r^2,$$
and by definition, 
$$
\ell_{\T}(B_r) \ge \sum_{u \in \eta_{R_I} } 
\sum_{v \in \eta^u_{r/2} }  |\mathcal T^v(B_r)|. 
$$ 
Thus by an application of the weak law of large numbers, and by taking the constant $C$ in the definition of $I$ large enough, we deduce
$$
\mathbb P(\ell_{\T}(B_r)>t) \geq \frac{c}{r^2} \exp(- \kappa' I),$$
for some (possibly different) positive constants $c$ and $\kappa'$. This concludes the proof of the lower bound in Theorem~\ref{theo.main}, in dimension four. \hfill $\square$ 

\noindent{\bf Acknowledgements.} We would like to thank Ofer Zeitouni 
for many stimulating and enthusiastic discussions at 
an early stage of this work. The authors acknowledge support from the grant ANR-22-CE40-0012 (project Local).

\end{document}